\documentclass[final,onefignum,onetabnum,nohypdvips]{siamart251216}
\usepackage{amsfonts}
\usepackage{amssymb}
\usepackage{array}
\usepackage{textcomp}
\usepackage{url}
\usepackage{graphicx}
\usepackage{booktabs}
\usepackage{multirow}
\usepackage{xcolor}
\usepackage{makecell}
\usepackage{subfigure}
\usepackage{arydshln}

\usepackage[algo2e,algosection,vlined,linesnumbered,resetcount,algoruled]{algorithm2e}
\SetKwProg{Function}{Function}{:}{}
\DontPrintSemicolon	

\numberwithin{table}{section}
\numberwithin{figure}{section}

\makeatletter
\newcommand{\sideeqtag}[1]{%
    \refstepcounter{equation}%
    \label{#1}%
    \tagform@{\theequation}%
}
\makeatother

\DeclareMathOperator*{\argmin}{arg\,min}
\DeclareMathOperator{\fl}{\operatorname{f\kern.2ptl}}
\DeclareMathOperator{\fllow}{\operatorname{f\kern.2ptl}_{\ell}}
\DeclareMathOperator{\flhigh}{\operatorname{f\kern.2ptl}_{h}}
\DeclareMathOperator{\flwork}{\operatorname{f\kern.2ptl}_{w}}

\DeclareMathOperator{\cond}{cond}

\def\ulow{u_{\ell}}
\def\uhigh{u_{h}}
\def\norm#1{\|#1\|_2}
\def\norm#1{\|#1\|}

\def\abs#1{|#1|}
\def\sse{\texttt{SSE}}

\newenvironment{revision}{\begingroup}{\endgroup}

\headers{Mixed-Precision k-means Algorithm}{E. Carson, X. Chen, and X. Liu}

\title{Mixed-Precision Lloyd's Algorithm for k-means Clustering\thanks{
			The first author acknowledges funding from the European Union (ERC, inEXASCALE, 101075632), and additionally acknowledges funding from the Charles University Research Centre program No. UNCE/24/SCI/005. The second author acknowledges funding from the France 2030 NumPEx Exa-MA (ANR-22-EXNU-0002) project managed by the French National Research Agency (ANR). 
            Views and opinions expressed are those of the authors only and do not necessarily reflect those of the European Union or the European Research Council. Neither the European Union nor the granting authority can be held responsible for them.
		}
}

\author{Erin Carson\thanks{Department of Numerical Mathematics, Charles University, Prague, Czech Republic
  (\email{carson@karlin.mff.cuni.cz}).
  }
\and Xinye Chen\thanks{
  LIP6, Sorbonne Université, CNRS, Paris, France
  (\email{xinye.chen@lip6.fr}).
  }
\and Xiaobo Liu\thanks{
  Max Planck Institute for Dynamics of Complex Technical Systems,
  Magdeburg, Germany
  (\email{xliu@mpi-magdeburg.mpg.de}).
  }
}

\ifpdf
\hypersetup{
  pdftitle={Mixed-Precision Lloyd's Algorithm for k-means Clustering},
  pdfauthor={Erin Carson, Xinye Chen, and Xiaobo Liu}
}
\fi

\begin{document}

\maketitle

\begin{abstract}
As low-precision arithmetic becomes more widely available on modern hardware, we develop mixed-precision variants of Lloyd's algorithm for k-means clustering, with a family of mixed-precision kernels for Euclidean distance computation as the main ingredient.  These kernels are guided by rounding-error analysis and use a simple reliability test to decide whether the expanded distance formula can be evaluated safely with low precision or a higher-precision correction by the direct distance formula is required. Thus, most distance computations can be carried out with low precision, while high-precision arithmetic is used selectively when cancellation may lead to a loss of accuracy. We evaluate the proposed methods on large-scale distance-computation benchmarks, synthetic clustering problems, and image-segmentation tasks. The experiments verify that the mixed-precision kernels on GPUs can substantially improve performance while retaining accuracy and convergence behavior of higher-precision baselines. In particular, our implemented distance computation in CUDA achieves speedups of up to $4\times$ over the \texttt{PyTorch} baseline that evaluates the expanded formula in FP64 on an NVIDIA A100 GPU, with overall maximum relative errors of order $10^{-3}$, while providing increased numerical robustness in cancellation-prone calculations.  The resulting mixed-precision kernels show efficiency for clustering and image segmentation, although the observed gains depend on factors such as  feature dimension, and number of clusters. The mixed-precision kernels offer a favorable trade-off between performance and accuracy for clustering, and suggest that similar ideas may be beneficial for other distance-based machine learning methods.
\end{abstract}

\begin{keywords}
Euclidean distance, k-means, mixed-precision algorithm, low precision, GPU computing, clustering
\end{keywords}

\begin{MSCcodes}
65G50, 62H30, 68T05, 68W10
\end{MSCcodes}

\section{Introduction}
The k-means algorithm has been one of the most studied algorithms in data mining and machine learning for decades, and it remains widely used due to its simplicity. Given a prescribed number of clusters, it seeks a partition of the data in which points are assigned to representative centers according to their similarity in a certain metric. The k-means algorithm plays a critical role in vector quantization \cite{gegr91}, bioinformatics~\cite{joka09}, computer vision \cite{chsu23}, anomaly detection \cite{gpb07}, database management \cite{orom04}, and nearest neighbor search \cite{jds10}. As stated in \cite{arva06}, with superpolynomial runtime, k-means is often very slow in practice. Many techniques have therefore been proposed to enhance the k-means algorithm regarding speed and scalability, e.g., parallelization and batch computations; associated variants include \texttt{k-means++}~\cite{arva07}, k-means$||$~\cite{bmvk12}, mini-batch k-means~\cite{nefl16}, and parallel k-means~\cite{frsi19}. 

\begin{table}[t]
	\caption{Parameters of four floating-point formats. Here $u=2^{-t}$ denotes the unit roundoff for rounding to nearest, $x_{\min}$ and $x_{\max}$ denote the smallest positive normalized and the largest finite floating-point numbers, respectively, $t$ is the number of binary significand bits (including the implicit leading bit), and $e_{\min}$ and $e_{\max}$ are the corresponding exponents.}
	\label{table:float-arith-parameter}
	\centering 
	\scalebox{.85}{
		\begin{tabular}{c l l l r r r}
			\toprule
			Format & $u$ & $x_{\min}$ & $x_{\max}$ & $t$ & $e_{\min}$ & $e_{\max}$\\ 
			\midrule 
			binary16 (FP16) & $2^{-11}\approx4.88\times10^{-4}$ & $6.10\times 10^{-5}$ & $6.55\times 10^{4}$& 11 & -14 & 15 \\
			bfloat16\phantom{i} (BF16) & $2^{-8\phantom{1}}\approx3.91\times10^{-3}$ & $1.18\times 10^{-38}$ & $3.39\times 10^{38}$ & 8 & -126 & 127 \\
			binary32 (FP32) & $2^{-24}\approx5.96\times10^{-8}$ & $1.18\times 10^{-38}$ & $3.40\times 10^{38}$ & 24 & -126 & 127 \\
			binary64 (FP64) & $2^{-53}\approx1.11\times10^{-16}$& $2.23\times 10^{-308}$ & $1.80\times 10^{308}$ & 53 & -1022 & 1023 \\
			\bottomrule
		\end{tabular}
	}
\end{table}

The increasing availability of lower-precision floating-point arithmetic, extending beyond the IEEE Standard 64-bit double precision (FP64) and 32-bit single precision (FP32) formats~\cite{ieee08re}, has motivated extensive research on low-precision number formats and arithmetic operations. In particular, 16-bit formats such as FP16 and BF16~\cite{ieee08re} are now widely supported in modern hardware and software environments; see \tablename~\ref{table:float-arith-parameter} for the parameters of the floating-point formats used in this paper.
By providing higher arithmetic throughput, reducing data movement, lowering energy consumption, and decreasing storage requirements, low-precision arithmetic has become an important tool for improving computational efficiency. It has been successfully exploited in a wide range of applications, including numerical linear algebra~\cite{aabc21, hima22}, unsupervised learning~\cite{cbd15}, deep neural networks~\cite{mnad18, rak23}, and other computational science and machine learning tasks. 
However, the use of low precision also introduces larger rounding errors, and rounding error analysis is therefore essential for identifying which parts of an algorithm can safely be performed in low precision~\cite{high:ASNA2}. Such analysis helps determine where mixed precision can be used judiciously, how computational resources can be exploited most effectively, and which precisions are needed to preserve the quality of the computed outputs.

Despite the prevalence of mixed-precision computing, it appears that its potential has largely remained unexploited in Euclidean distance computation and k-means clustering. 
In this work, we develop a mixed-precision variant of Lloyd's algorithm for $k$-means clustering and provide CUDA implementations of several algorithmic variants. 
Guided by rounding error analysis, we propose a mixed-precision approach to Euclidean distance computation.  Algorithm~\ref{alg:mp-dist} evaluates distances in low precision and selectively recomputes them in higher precision. Its GPU realization combines low-precision input and intermediate storage with FP32 accumulation; the corresponding error model and GPU arithmetic paths are described in Sections~\ref{sec:accum-storage} and~\ref{sec:gpu-arithmetic}, respectively.

To the best of our knowledge, this work is the first to systematically  analyze the numerical stability of two common formulas for computing Euclidean distances, their evaluation in mixed precision, and their incorporation into a mixed-precision Lloyd's k-means algorithm. The resulting mixed-precision k-means has already been integrated into the symbolic time series aggregation library \texttt{fABBA}~\cite{10.1145/3532622} that supports large-scale time series and signal processing, as well as downstream applications in large language models~\cite{11397443}.

The paper is organized as follows. In Section~\ref{sect:related-work} we review some recent work on k-means clustering with extra computing resources such as GPU and parallel computing.
Lloyd's k-means algorithm is then presented in Section~\ref{sect:kmeans}, where we briefly discuss two Euclidean distance computation formulas. 
In Section~\ref{sec:round} we discuss the numerical stability of Lloyd’s k-means method, including the stages of distance computation and center update. In particular, we perform rounding error analyses for two distance computation formulas and derive the new mixed-precision method. The resulting mixed-precision k-means method is presented in Section~\ref{sect:mp-kmeans}. Numerical experiments on variants of the mixed-precision distance computation method and k-means clustering are presented in Section~\ref{sec:numer-exper}. Conclusions are drawn in Section~\ref{sec:conclusion}.

Throughout the paper, inequalities between vectors hold componentwise, and the norm $\norm{\cdot}$ denotes the $\ell^2$ norm $\norm{x}=(\sum_i|x_i|^2)^{1/2}=\sqrt{x^Tx}$.
We use $|S|$ to represent the cardinality of a set $S$. 

\section{Related work}\label{sect:related-work}
In this section, we briefly review several works on accelerating k-means clustering that are closely related to our work, with particular emphasis on approaches that improve performance by exploiting additional hardware or parallel computing resources.
The k-means algorithm as well as the trending seeding of $D^2$ weighting~\cite{arva07} are known to be inherently sequential, which makes it tricky to implement the algorithm in a parallel way. 
Numerous novel approaches have been proposed to enhance scalability and speed. The authors of~\cite{bel13} proposed a distributed computing scheme for coreset construction that enables an acceptable low communication complexity and distributed clustering. Moreover, an initialization approach of $D^2$ weighting with MapReduce is proposed in~\cite{xqlm14} to enable the use of only one job for choosing the $k$ centers, which can significantly reduce the communication and I/O costs. 
A k-means algorithm based on GPUs using single instruction multiple data architectures is presented in~\cite{bhol09}; the algorithm performs the assignment of data points as well as cluster center updates on the GPU, which offers acceleration by orders of magnitude. The GPU power is also harnessed in~\cite{lbrz18} to update centroids efficiently; the k-means algorithm therein also executes a single-pass strategy to remove data transfers{{--}}up to 2 and 18.5 times higher throughput is achieved compared to multi-pass and cross-processing strategies, respectively. 
More recently, Li et al.~\cite{lfp23} proposed a batched GPU-based k-means algorithm, which achieves substantially faster performance than a standard GPU-based implementation while generating outputs of comparable quality.

\section{The k-means method}\label{sect:kmeans}
As a classical problem in machine learning and computational geometry,  k-means clustering seeks a codebook of $k$ vectors, i.e., $C=[c_1, \ldots, c_k] \in \mathbb{R}^{r \times k}$ in $n$ vectors $P=[p_1, \ldots, p_n] \in \mathbb{R}^{r\times n}$, such that $k \ll n$, where each $c_i$ is associated with a unique cluster $S_i \subseteq P$.  Letting $S_1, S_2, \ldots, S_k \subseteq P$ be the $k$ computed clusters, k-means aims to minimize the sum of squared errors, given by
\begin{equation}\label{eq:sse}
    \sse = \sum_{i=1}^{k} \phi (c_i, S_i),
\end{equation}
where $\phi$ denotes some energy function and $c_i$ denotes the center of cluster $S_i$. In this paper, $\phi (c_i, S_i) = \sum_{p\in S_i}d(p, c_i)^2$, where $d(p_i, p_j)$ denotes the Euclidean distance between $p_i$ and $p_j$, measured by $\norm{p_i - p_j} =\sqrt{(p_i-p_j)^T(p_i-p_j)}$.
Lloyd's algorithm \cite{lloy82}, also known as the \(k\)-means algorithm, is an iterative method for vector quantization that aims to minimize the sum of squared errors (\sse), without guaranteeing a globally optimal solution.
In exact arithmetic, Lemma~\ref{eq:converge} shows that choosing the mean $\mu_i=\sum_{p\in S_i}p/|S_i|$ as the center of each nonempty cluster cannot increase the \sse{} in~\eqref{eq:sse}. Together with nearest-center assignment, this makes the \sse{} nonincreasing across iterations of Lloyd's algorithm (Algorithm~\ref{alg:kmeans}).

\begin{lemma}[{\cite[Lem.~2.1]{arva07}}]\label{eq:converge}
	Given an arbitrary data point $p$ in the cluster $S$ whose mean center is 
	denoted by $\mu$, we have
	\begin{equation}\label{eq:center_benefit}
		\phi (p,S) =  \phi (\mu, S) + |S| d(p, \mu)^2.
	\end{equation}
\end{lemma}

\subsection{Lloyd's k-means algorithm}\label{sect:lloyd-kmeans}

\begin{algorithm2e}[!t]
	\caption{Lloyd's k-means algorithm.}
	\label{alg:kmeans}
	\KwIn{Data set $P=[p_i]_{i=1}^{n}$, number of clusters $k$,
		convergence tolerance $\tau>0$, and maximum number of iterations $t_{\max}$.}
	\KwOut{Centers $\{c_1,\ldots,c_k\}$ and cluster assignments.}
	Normalize dataset $P$\;
	Select initial centers $\{c_1^{(0)},\ldots,c_k^{(0)}\}$\;
	\For{$t \gets 1,\ldots,t_{\max}$}{
		\ForEach{$p_i \in P$}{
			\ForEach{$j=1,\ldots,k$}{
				$d_{ij} \gets \norm{p_i-c_j^{(t-1)}}^2$\;\label{alg:k-means:line:forloop-start}
			}
			Assign $p_i$ to cluster $S_j^{(t)}$, where
			$j=\argmin_{s} d_{is}$\; \label{alg:k-means:line:forloop-end}
		}
		\ForEach{$j=1,\ldots,k$}{
			$c_j^{(t)} \gets 
			\frac{1}{|S_j^{(t)}|}\sum_{p_i\in S_j^{(t)}} p_i
			$\;
		}
		\If{$\max_j \norm{c_j^{(t)}-c_j^{(t-1)}} < \tau$}{
			\textbf{break}\;
		}
	}
	Recompute all distances and assign each $p_i$ to its closest center\;
	\Return final centers $\{c_1,\ldots,c_k\}$ and cluster assignments
\end{algorithm2e}

It is well known that computing a globally optimal k-means clustering is NP-hard.  In practice, therefore, one usually resorts to heuristic algorithms, the most classical of which is Lloyd's algorithm~\cite{lloy82}, often also referred to as the k-means algorithm~\cite{gray84}. Lloyd's k-means algorithm (given as Algorithm~\ref{alg:kmeans}), is a local improvement heuristic: starting from an initial set of cluster centers, it alternates between assigning each data point to its nearest center and updating each center as the mean of the points assigned to it. Although this procedure is simple and widely used, it may converge to a local minimum, has no general worst-case approximation guarantee~\cite{kmnp02}, and can have superpolynomial complexity in the worst case.

The quality and convergence behavior of Lloyd's algorithm depend strongly on the choice of the initial centers. A particularly successful initialization scheme is seeding by $D^2$ weighting, introduced by Arthur and Vassilvitskii~\cite{arva07}. The resulting method, commonly known as \texttt{k-means++}, is $O(\log k)$-competitive in expectation with the optimal clustering and is the default choice for the k-means implementation of \texttt{scikit-learn} \cite{pvgm11}{{--}}the popular,  comprehensive machine learning library. A parallel alternative is \texttt{k-means$\|$}~\cite{bmvk12}, which replaces the sequential $D^2$-sampling step in \texttt{k-means++} by an oversampling-based parallel procedure followed by a reclustering step.

We do not study initialization schemes in this paper, since our objective is to study the effect of mixed-precision arithmetic on the computationally dominant steps of Lloyd's iteration, in particular, the Euclidean distance computations. 
We therefore use a simple randomized initialization, selecting $k$ initial centers uniformly from the data, and apply the same initialization across the methods being compared. More sophisticated seeding strategies, such as \texttt{k-means++} or \texttt{k-means$\|$}, are orthogonal to the mixed-precision techniques developed here and could be combined with them in a complete implementation.

\subsection{Distance computation}\label{sec:intro-dist-comput}
The computationally dominant step in the k-means algorithm, even including cluster initialization, is the computation of pairwise distances.
Two natural approaches to forming the squared Euclidean distance $d(x, y)^2$ between two points $x,y\in\mathbb{R}^r$ are
\begin{center}
\begin{tabular}{@{}r@{\quad}l@{\hspace{1.5em}}r@{\quad}l@{}}
    \sideeqtag{eq:dist-direct} &
    $\displaystyle d(x,y)^2 = (x-y)^T(x-y),$ &
    \sideeqtag{eq:dist-expanded} &
    $\displaystyle d(x,y)^2 = x^Tx - 2x^Ty + y^Ty.$
\end{tabular}
\end{center}
The second scheme~\eqref{eq:dist-expanded} is of particular computational interest and is the prevailing choice in algorithmic implementations, for example, in the \texttt{k-means} algorithm of the machine learning library \cite{pvgm11}. In \texttt{PyTorch}~\cite{paszke2019pytorch}, \texttt{torch.cdist}\footnote{https://docs.pytorch.org/docs/2.11/generated/torch.functional.cdist.html} by default uses~\eqref{eq:dist-direct} to compute pairwise Euclidean distances for small inputs but switches to a matrix-multiplication-based formulation related to~\eqref{eq:dist-expanded} for larger inputs; this behavior can be controlled explicitly through the \texttt{compute\_mode} argument.

If the distance formula~\eqref{eq:dist-direct} is used, it requires about $kn$ distance computations in each iteration of executing Lines~\ref{alg:k-means:line:forloop-start}--\ref{alg:k-means:line:forloop-end} of Algorithm~\ref{alg:kmeans} since each of the $n$ points $p_i$ must be compared with each of the $k$ centers $c_j$. Each distance evaluation costs about $3r$ floating-point operations (flops), and so the distance computation in one iteration costs about $3r kn$ flops, summing to an overall cost of $3Tr kn$ flops for $T$ iterations.

In contrast, the expanded formula~\eqref{eq:dist-expanded} allows the $n$ inner products $p_i^Tp_i$, $i=1\colon n$, to be precomputed and stored. This costs approximately $2r n$ flops and these values can then be reused across all iterations.
The center norms $c_j^Tc_j$ needs be recomputed whenever the centers are updated, costing approximately $2rk$ flops per iteration.  
The dominant cost is the computation of all inner products $p_i^Tc_j$, which
can be performed at a cost of $2rkn$ flops per iteration. Forming the final distance matrix from these terms costs only $O(kn)$ additional flops. Therefore,  the overall cost of distance computing via~\eqref{eq:dist-expanded} in a total of $T$ iterations is about $2rn + T(2rk+2rkn+3kn)$ flops.
The leading term is approximately $2Tr kn$ flops if $n\gg 1$ and $r\gg 1$, which is the case in practical applications.

More importantly, the expanded formula~\eqref{eq:dist-expanded} avoids duplicated computations by precomputing the squared norm values and enables formulation of the main computation (i.e., the inner products between data points and centres) as a level-3 basic linear algebra subprograms (BLAS) matrix--matrix multiplication~\cite{bddd02}. 
Collecting the scalar cross terms $x_i^Ty_j$ for all pairs forms the matrix $XY^T$, as we now explain.
Let
\begin{equation*}
X =
\begin{bmatrix}
x_1 & x_2 & \ldots & x_m
\end{bmatrix}^T
\in \mathbb{R}^{m\times r},
\qquad
Y =
\begin{bmatrix}
y_1 & y_2 & \ldots & y_n
\end{bmatrix}^T
\in \mathbb{R}^{n\times r},
\end{equation*}
where \(x_i,y_j\in\mathbb{R}^{r}\). Then the squared pairwise Euclidean distance matrix \(D\in\mathbb{R}^{m\times n}\), defined by
\begin{equation*}
D_{ij} = \|x_i-y_j\|_2^2 =
x_i^T x_i - 2x_i^T y_j + y_j^T y_j,
\qquad i=1,\ldots,m,\quad j=1,\ldots,n,
\end{equation*}
can be written as
\begin{equation}\label{eq:dist-blas3}
D =
p\mathbf{1}_n^T - 2XY^T + \mathbf{1}_m q^T,
\end{equation}
where \(p_i=x_i^T x_i\), \(q_j=y_j^T y_j\), and \(\mathbf{1}_m\in\mathbb{R}^{m}\), \(\mathbf{1}_n\in\mathbb{R}^{n}\) denote the vectors of all-ones. This formulation evaluates the squared norms of \(x_i\) and \(y_j\) only once and reuses the values across all corresponding pairs. The computationally dominant term \(-2XY^T\) can be evaluated by a GEMM,
\begin{equation*}
D \leftarrow \alpha AB+\beta C,
\end{equation*}
with $A=X, B=Y^T, C=p \mathbf{1}_n^T + \mathbf{1}_m q^T, \alpha=-2$, and $\beta=1$, so the formulation~\eqref{eq:dist-blas3} can give substantially better performance on modern hardware than implementing~\eqref{eq:dist-direct} via level-2 BLAS.

\begin{table}
\caption{Runtime comparison of distance formulas~\eqref{eq:dist-direct} and~\eqref{eq:dist-expanded}.}
\label{tab:dist-formula}
\centering
\renewcommand{\arraystretch}{1.12}\resizebox{0.73\linewidth}{!}{%
    \begin{tabular}{l l r r r r}
        \toprule[1pt]
        & & \multicolumn{2}{c}{FP32} & \multicolumn{2}{c}{FP64} \\
        \cmidrule(lr){3-4} \cmidrule(lr){5-6}
        Data distribution & Formula & Time (ms) & Speedup & Time (ms) & Speedup \\
        \midrule
        \multirow{2}{*}{Random} & Direct~\eqref{eq:dist-direct} 
        & 30.64 & \multirow{2}{*}{$16.7 \times$} 
        & 57.84 & \multirow{2}{*}{$21.1 \times$} \\
        & Expanded~\eqref{eq:dist-expanded} & 1.84 &  & 2.74 & \\
        \midrule
        \multirow{2}{*}{Nearly located}	& Direct~\eqref{eq:dist-direct} 
        & 30.60 & \multirow{2}{*}{$17.0 \times$} 
        & 57.68 & \multirow{2}{*}{$21.3 \times$} \\
        & Expanded~\eqref{eq:dist-expanded} & 1.80 &  & 2.71 & \\
        \bottomrule[1pt]
    \end{tabular}}
\end{table}

To give a brief illustration of the efficiency of the two distance formulas, we computed distances between two 
sets of 5,000 vectors in $\mathbb{R}^{128}$ (so $m=n=5,000$ and $r=128$), which requires a total of 25 million distance entries. 
\tablename~\ref{tab:dist-formula} presents the results, where we have tested two different data-distribution settings: randomly generated data and nearly located data; see Section~\ref{sec:numer-exper} for full details on the experiment environment. The expanded formula~\eqref{eq:dist-expanded} is clearly much faster than the direct formula~\eqref{eq:dist-direct} in both FP32 and FP64. 

\section{Numerical stability of the k-means method}\label{sec:round}
In this work, numerical stability is understood in the standard sense of backward and forward error analysis~\cite[sect.~1.5]{high:ASNA2}.
To study the numerical stability of the k-means method,
we need to look at its two main computational steps: distance computation and center update.
We will use the standard model of floating-point arithmetic introduced in~\cite[sect.~2.2]{high:ASNA2} for our stability analysis:
\begin{equation}\label{eq:fpmodel}
    \fl(x \hspace{2pt}\texttt{op}\hspace{2pt} y) = (x \hspace{2pt}\texttt{op}\hspace{2pt} y) (1 + \delta), \quad  |\delta| \le u,
\end{equation}
where $u$ denotes the unit roundoff associated with the floating-point number system, $x$ and $y$ are floating-point numbers, and \texttt{op} denotes
addition, subtraction, multiplication, or division.
We assume rounding to nearest in the error model and do not consider overflow and underflow. It is understood that inputs to the working precision are exactly representable.
For inner products of two vectors, we have~\cite[sect.~3.1]{high:ASNA2},
\begin{equation}\label{eq:fpmode-vecprod}
	\fl(x^Ty) = x^Ty + s,\quad |s|\le \gamma_r|x|^T|y|,
 \end{equation}
where $\gamma_r:= r u/(1-r u)$ and $r u<1$. Equation~\eqref{eq:fpmodel} describes the rounding error of a single elementary arithmetic operation, whereas equation~\eqref{eq:fpmode-vecprod} bounds the accumulated error of an inner product evaluated sequentially under the same floating-point assumptions, and it does not represent a single rounding of the exact inner product.
Repeated application of the elementary-operation model yields the factor $\gamma_r$. The resulting absolute error bound scales with $|x|^T|y|$, rather than $|x^Ty|$, so cancellation between terms can amplify the relative error.

In addition to the working precision, with unit roundoff $u$, we consider two floating-point arithmetics with unit roundoffs $\ulow$ and $\uhigh$ that satisfy
\begin{equation*}
	0 < \uhigh \le u \le \ulow < 1.
\end{equation*}
We refer to these two arithmetics as the low and high precisions, respectively. By ``computing in precision $u_*$'', we mean that $u_*$ is the unit roundoff of the current precision of computation, where $u_*=\uhigh$, $u$, or $\ulow$. The operator $\fl(\cdot)$ with a subscript denotes the computation in the associated precision; when the subscript is omitted, the computation is understood to be performed in the working precision. Computed quantities are denoted by a hat.

The process of computing vector inner products is backward stable~\cite[eq.~(3.4)]{high:ASNA2} but the forward accuracy depends also on the conditioning.  Consider the inner product $f(x):=x^Ty$, where $y$ is a fixed nonzero vector and $x^Ty\ne 0$. This function has the normwise relative condition number~\cite[sect.~3.1]{high:FM} (in the $2$-norm)
\begin{equation*}
\cond(f,x) := \lim_{\varepsilon\to 0} \sup_{\norm{\Delta x}\le \varepsilon \norm{x}}  \frac{\abs{f(x+\Delta x) - f(x)}}{\varepsilon |f(x)|},
\end{equation*}
which is given explicitly by
\begin{equation}\label{eq:x^Ty-cond}
    \cond(f,x) = \frac{\norm{\nabla f }\norm{x}}{\abs{f(x)}} 
    = \frac{\norm{y}\norm{x}}{\abs{x^Ty}} = \frac{1}{|\cos\omega|},
\end{equation}
where $\omega\in[0,\pi]$ is the angle between $x$ and $y$. This shows that the vector dot product is more sensitive (has larger condition number) if the two vectors are close to being orthogonal. Therefore, the relative forward error in $\fl(x^Ty)$ can be large when $x$ and $y$ are close to being orthogonal. On the other hand, at $x=y\ne 0$, with $y$ still held fixed, $\cond(f,x) = 1$ and
$|\fl(x^Tx)-x^Tx|\le \gamma_r|x|^T|x|=\gamma_r x^Tx$, which guarantees high relative accuracy when $\gamma_r$ is small.

\begin{revision}
For a composite expression, $\fl_*(\cdot)$ denotes evaluation of all its elementary arithmetic operations in precision $u_*$, with the specified evaluation order; it shall not be interpreted as exact evaluation followed by one rounding. 
For example, a direct squared distance forms $z_j=\fl_*(x_j-y_j)$ and evaluates the inner product of the computed vector $z$ with itself.  A center update similarly consists of a floating-point sum followed by division.  Fused multiply-add (FMA) and the order in which products are accumulated require corresponding error models, e.g., FMA evaluates a product and an addition with a single rounding. The distinction between these reference models and the GPU execution paths is made explicit in Section~\ref{sec:accum-storage} and Section~\ref{sec:gpu-arithmetic}, respectively.
\end{revision}

\subsection{Distance formulas}\label{sec:numer-stab-dist-comput}
We begin by discussing the direct distance formula~\eqref{eq:dist-direct}. Its proof is a standard textbook exercise based on~\cite[sect.~3.1]{high:ASNA2}; a proof is also given in~\cite{chgu24b}.

\newcommand{\ddir}[1][]{\widehat{d}_{#1,\mathsf{dir}}}
\newcommand{\Dddir}[1][]{\Delta \widehat{d}_{#1,\mathsf{dir}}}
\newcommand{\dexp}[1][]{\widehat{d}_{#1,\mathsf{exp}}}
\newcommand{\Ddexp}[1][]{\Delta \widehat{d}_{#1,\mathsf{exp}}}

\begin{theorem}\label{thm:dist-diff}
Let $x, y \in \mathbb{R}^{r}$, and let $\ddir[r] = \fl((x-y)^T(x-y))$ be the squared distance computed via~\eqref{eq:dist-direct} in precision $u$. Then 
\begin{equation*}
\ddir[r] = d(x, y)^2 + \Dddir[r], \quad 
|\Dddir[r]| \le \gamma_{r+2} d(x, y)^2,
\end{equation*}
given that $(r+2) u < 1$.
\end{theorem}

Theorem~\ref{thm:dist-diff} shows that the formula~\eqref{eq:dist-direct} computes the squared distance to high relative forward accuracy when $\gamma_{r+2}$ is small. But the same \emph{cannot} be guaranteed for the other formula~\eqref{eq:dist-expanded}, as shown by the following result.

\begin{theorem}\label{thm:dist-exp}
Let $x, y \in \mathbb{R}^{r}$, and let $\dexp[r] = \fl( x^T x - 2 x^T y + y^T y)$ be the squared distance computed via \eqref{eq:dist-expanded} in precision $u$. Then
\begin{equation*}
    \dexp[r] = d(x,y)^2 + \Ddexp[r], \quad |\Ddexp[r]| \le 
\gamma_{r+2} (x^T x + 2|x|^T |y| + y^T y),
\end{equation*}
given that $(r+2) u < 1$.
\end{theorem}
\begin{proof}
Define
\begin{equation*}
\widehat p = \fl(x^Tx), \quad
\widehat q = \fl(x^Ty), \quad
\widehat s = \fl(y^Ty),
\end{equation*}
so we have $ \dexp[r] = \fl\bigl(\fl(\widehat p - 2\widehat q)+\widehat s\bigr)$ since multiplication by $2$ does not incur rounding in binary floating-point arithmetic. By~\eqref{eq:fpmode-vecprod},
\begin{equation*}
\widehat p = x^Tx(1+\theta_1), \quad 
\widehat s = y^Ty(1+\theta_2), \quad 
|\theta_1|, |\theta_2|\le \gamma_r,
\end{equation*}
and
\begin{equation*}
\widehat q = x^Ty + \Delta q, \quad |\Delta q|\le \gamma_r |x|^T|y|.
\end{equation*}
Moreover, there exist \(\delta_1,\delta_2\) with $|\delta_1|, |\delta_2| \le u$, such that $\dexp[r] = \bigl((\widehat p-2\widehat q)(1+\delta_1)+\widehat s\bigr)(1+\delta_2)$,
so
\begin{equation*}
\dexp[r] =
x^Tx(1+\theta_1)(1+\delta_1)(1+\delta_2) 
- 2(x^Ty+\Delta q)(1+\delta_1)(1+\delta_2) + y^Ty(1+\theta_2)(1+\delta_2).
\end{equation*}
Subtracting $d(x,y)^2 = x^Tx - 2x^Ty + y^Ty$ and using the standard
$\gamma$-notation product rule gives~\cite[Lem.~3.3]{high:ASNA2}
\begin{equation*}
|\dexp[r] - d(x,y)^2|
\le
\gamma_{r+2}x^Tx
+
2\gamma_{r+2}|x|^T|y|
+
\gamma_{r+2}y^Ty,
\end{equation*}
which is the desired bound.
\end{proof}

\begin{revision}
Let us fix $y$ and consider the distance function
\begin{equation*}
    g(x) = d(x,y)^2 = x^Tx - 2x^Ty + y^Ty
    = (x-y)^T(x-y).
\end{equation*}
For $x\ne y$, the gradient is $\nabla g(x) = 2(x-y)$, and hence the normwise relative condition number is 
\begin{equation*}
    \cond(g,x) = \frac{\norm{\nabla g }\norm{x}}{\abs{g(x)}} 
    = \frac{2\norm{x-y}\norm{x}}{\norm{x-y}^2} = \frac{2\norm{x}}{\norm{x-y}}.
\end{equation*}
This shows that the distance computation is intrinsically ill-conditioned near $x=y$ (in exact arithmetic). 
This sensitivity to perturbation can be exposed by rounding errors in floating-point arithmetic, and
the two bounds in Theorem~\ref{thm:dist-diff} and Theorem~\ref{thm:dist-exp} can exhibit different numerical behavior. 
For two representable inputs $x\ne y$, the former gives a relative error bound $\gamma_{r+2}$, whereas the latter gives
\begin{equation*}
\frac{|\dexp[r]-d(x,y)^2|}{d(x,y)^2}
\le \gamma_{r+2}
\frac{\|x\|^2+2|x|^T|y|+\|y\|^2}{\|x-y\|^2}.
\end{equation*}
For nearby nonzero vectors the extra factor can become arbitrarily large. 
The expanded formula evaluates $\|x\|^2+\|y\|^2-2x^Ty$: when $x\approx y$, the norm sum and twice the inner product are both much larger than their difference. 
Rounding errors in these terms can therefore be large relative to the resulting squared distance, leading to \emph{catastrophic cancellation}. 
The direct formula instead forms the componentwise differences $x_\ell-y_\ell$ and sums their squares, and this computation avoids cancellation between the separately computed norm and cross-product terms.

\end{revision} 

\subsection{Distance computation with reliability test}\label{sec:numer-stab-fallback}
Our discussion in Section~\ref{sec:numer-stab-dist-comput} shows that the commonly used distance formula~\eqref{eq:dist-expanded} can be unreliable when the two vectors are close, due to severe cancellation.
This potential instability can be particularly problematic in low-precision arithmetic, where rounding errors are larger and can accumulate more significantly.
Motivated by the error analysis, we design a numerical safeguard that compares the computed distance with an easily computable error bound, and, when necessary, switches to the distance formula~\eqref{eq:dist-direct}, which is less efficient but more reliable in computations that are prone to cancellation.

\begin{corollary}
For $x, y\in\mathbb{R}^{r}$, the computed squared distance $\dexp[r] = \fl( x^T x - 2 x^T y + y^T y )$ in precision $u$ satisfies
\begin{equation}\label{eq:expanded-abs-bound}
|\dexp[r] - d(x,y)^2| \le
2\gamma_{r+2} \bigl(\norm{x}^2 + \norm{y}^2\bigr),
\end{equation}
given that $(r+2) u < 1$.
\end{corollary}

\begin{proof}
By the Cauchy--Schwarz inequality, $|x|^T|y| \le  \norm{x} \norm{y}$, which is bounded above by $\frac{1}{2} \bigl(\norm{x}^2 + \norm{y}^2\bigl)$. Therefore,
$$
x^Tx + 2|x|^T|y|+ y^Ty \le 2\bigl(\norm{x}^2 + \norm{y}^2\bigl).
$$
Substituting this expression into Theorem~\ref{thm:dist-exp} yields the desired bound.
\end{proof}

The quantity 
\begin{equation}\label{eq:error-computable}
E(x, y) := 2\gamma_{r + 2} \bigl( \norm{x}^2 + \norm{y}^2 \bigr)\ge 0
\end{equation}
in~\eqref{eq:expanded-abs-bound} is easily computable and so provides a convenient way of checking the absolute error in the computed squared distance $\dexp[r]$. 
If $\dexp[r] \ge \rho E(x, y)$ for some $\rho> 1$, we have
\begin{equation*}
d(x,y)^2  \ge \dexp[r] - |\dexp[r] - d(x,y)^2|
\ge \dexp[r] - E(x,y) \ge (\rho-1)E(x,y).
\end{equation*}
It follows that
\begin{equation}\label{eq:expanded-rel-bound}
\frac{| \dexp[r] - d(x,y)^2|}{|d(x,y)^2|}
\le \frac{E(x, y)}{(\rho-1)E(x,y)} = \frac{1}{\rho-1}.
\end{equation}
The parameter $\rho>1$ can be viewed as a \emph{safety factor} that controls the \emph{relative} error in $\dexp[r]$.

Therefore, we propose to exploit the more efficient distance formula~\eqref{eq:dist-expanded} with the fallback rule
\begin{equation} \label{eq:fallback-rule}
\widehat d(x,y)^2 = \begin{cases}
\dexp[r], & \text{if}\; \dexp[r] > \rho E(x, y),\\ 
\ddir[r], & \text{otherwise},
\end{cases}
\end{equation}
where $\widehat d(x,y)^2$ presents the value we finally take as the computed squared distance between $x$ and $y$. 
When the reliability test $\dexp[r] > \rho E(x, y)$ fails (so there is no control over the relative error in $\dexp[r]$, signifying possible occurrence of significant cancellation), the fallback rule distrusts the computed distance and switches to computing the squared distance via the less efficient but numerically stable formula~\eqref{eq:dist-direct}.

\newcommand{\x}{\widetilde x}
\newcommand{\y}{\widetilde y}
\newcommand{\Dx}{\Delta x}
\newcommand{\Dy}{\Delta y}

\newcommand{\qddif}[1][]{\widehat{d}^{\ell}_{#1,\mathsf{dif}}}
\newcommand{\qdexp}[1][]{\widehat{d}^{\ell}_{#1,\mathsf{exp}}}
\newcommand{\wtrho}{\widetilde \rho}

\subsection{Effect of input quantization}\label{sec:numer-stab-fallback-mp}
Thus far, we have been considering the case where the arithmetic of computation is uniform with precision $u$. In our mixed-precision setting of interest, the original data will be converted into a lower precision so that the subsequent pairwise distance computation is performed in the low precision $\ulow > u$.
In this case we have to quantify an additional source of error from the input quantization, i.e., converting the working-precision inputs to the distance formulas into lower precision. We assume $(r+2)\ulow<1$ throughout this subsection.

We denote $\x := \fllow(x)$ and $\y := \fllow(y)$ as the low-precision data points \emph{converted} from $x$ and $y$, respectively, so the differences $\Dx := \x - x$ and $\Dy := \y - y$ satisfy the componentwise bounds
\begin{equation}\label{eq:quantis-compon-bound}
   |\Dx| \le \ulow|x|, \quad 
   |\Dy| \le \ulow|y|.
\end{equation}
Now define
$$
    e :=(\x-x) - (\y-y) = \Dx - \Dy.
$$
Then
$$
\x - \y = (x - y) + e, \quad
|e| \le \ulow (|x| + |y|).
$$
We have $d(\x, \y)^2 - d(x, y)^2 = 2(x-y)^T e+e^T e$, and so 
$$
|d(\x, \y)^2 - d(x, y)^2 | \le  C(x,y),
$$ 
where $C(x,y)$ is the conversion error bound defined by
\begin{equation}\label{eq:quantis-convert-error}
    C(x,y) := 2\ulow |x-y|^T (|x|+|y|) + \ulow^2 (|x|+|y|)^T(|x|+|y|).
\end{equation}

Invoking Theorem~\ref{thm:dist-diff}, the computed distance (in precision $\ulow$) between the quantized data via formula~\eqref{eq:dist-direct}, which we denote as $\qddif[r]$, satisfies 
\begin{equation*}
|\qddif[r] - d(\x, \y)^2| \le \gamma^{\ell}_{r+2} d(\x, \y)^2.
\end{equation*}
For $x\ne y$, using the triangle inequality $|\qddif[r] - d(x, y)^2|\le |\qddif[r] - d(\x, \y)^2| + |d(\x, \y)^2 - d(x, y)^2|$, the relative discrepancy to the original distance can be bounded above by
\begin{equation*}
    \frac{|\qddif[r] - d(x, y)^2|}{d(x, y)^2}  \le 
    (1 + \gamma^{\ell}_{r+2})\frac{C(x,y)}{d(x, y)^2} + \gamma^{\ell}_{r+2}.
\end{equation*}
Define
\begin{equation*}    
\eta:= \frac{\ulow (\norm{x} + \norm{y})}{\norm{x-y}},
\quad x\ne y.
\end{equation*}
Since
\begin{equation*}
    C(x,y)  \le 2\ulow \norm{x-y} (\norm{x} + \norm{y})
 + \ulow^2 (\norm{x} + \norm{y})^2  = (2\eta + \eta^2)d(x, y)^2, 
\end{equation*}
we arrive at the relative error bound
\begin{equation}\label{eq:quantis-rel-bound-dif}
    \frac{|\qddif[r] - d(x, y)^2|}{d(x, y)^2}  \le 
    (2\eta + \eta^2)(1 + \gamma^{\ell}_{r+2}) + \gamma^{\ell}_{r+2}.
\end{equation}
\begin{revision}
Together with a small $\gamma^\ell_{r+2}$, the condition $\eta\ll1$, equivalently
\begin{equation}\label{eq:quantis-necess-cond}
    \norm{x-y}\gg\ulow(\norm{x}+\norm{y}),
\end{equation}
is sufficient for the preceding worst-case bound to be small. However, when nearby inputs are rounded to the same low-precision vector, no subsequent arithmetic on those quantized inputs can recover their original nonzero distance.  We therefore stress that this information loss is irreversible; it cannot be recovered by either distance formula.
\end{revision}

Now consider, on the other hand, the computed distance (in precision $\ulow$) between the quantized data via formula~\eqref{eq:dist-expanded}. By Theorem~\ref{thm:dist-exp}, we have
\begin{equation*}
|\qdexp[r] - d(\x, \y)^2 | \le \gamma^{\ell}_{r+2} (\x^T \x + 2|\x|^T |\y| + \y^T \y).
\end{equation*}
Define
$$
E^{\ell}(\x,\y) := 2\gamma^{\ell}_{r+2} (\norm{\x}^2 + \norm{\y}^2)
\ge |\qdexp[r] - d(\x, \y)^2|,
$$
corresponding to the computable~\eqref{eq:error-computable} in the uniform precision.
We can obtain from~\eqref{eq:quantis-compon-bound} that $\norm{x}^2 + \norm{y}^2 \le (1-\ulow)^{-2} (\norm{\x}^2 + \norm{\y}^2)$, and so
\begin{equation}\label{eq:quantis-error-computable}
    E^{\ell}(\x,\y) \ge  2\gamma^{\ell}_{r+2}(1-\ulow)^2 (\norm{x}^2 + \norm{y}^2).
\end{equation}

Suppose 
\begin{equation}\label{eq:quantis-reliability-test}
    \qdexp[r] > \wtrho E^{\ell}(\x,\y), \quad \wtrho>2.
\end{equation}
Similarly to the derivation of~\eqref{eq:expanded-rel-bound} in Section~\ref{sec:numer-stab-fallback}, we can obtain
\begin{equation*}
    d(\x,\y)^2 \ge (\wtrho-1)E^{\ell}(\x,\y)
\end{equation*}
and
\begin{equation}\label{eq:quantis-expanded-rel-bound}
\frac{| \qdexp[r] - d(\x,\y)^2|}{|d(\x,\y)^2|} \le \frac{1}{\wtrho-1}.
\end{equation}
Using $\norm{|x-y|}\le \norm{|x| + |y|} \le \norm{x} + \norm{y}$ and then the bound~\eqref{eq:quantis-error-computable}, we now bound the conversion error~\eqref{eq:quantis-convert-error} by
\begin{equation*}
    C(x,y)  
    \le 2(2\ulow + \ulow^2)(\norm{x}^2 + \norm{y}^2) 
      \le   \frac{2\ulow+\ulow^2}{(1-\ulow)^2\gamma^\ell_{r+2}}E^{\ell}(\x,\y).
\end{equation*}
For $r\ge2$ and $\ulow\le1/8$ (satisfied by all formats in Table~\ref{table:float-arith-parameter}), $\gamma^\ell_{r+2}\ge(r+2)\ulow$ yields
\begin{equation}\label{eq:quantis-convert-error-bound}
    C(x,y) \le \frac{2+\ulow}{(r+2)(1-\ulow)^2}E^{\ell}(\x,\y) \le E^{\ell}(\x,\y),
\end{equation}
which implies
\begin{equation}
    d(x, y)^2\ge d(\x, \y)^2 - C(x,y) \ge (\wtrho -2)E^{\ell}(\x,\y).
\end{equation}
Finally, using the triangle inequality $|\qdexp[r] - d(x, y)^2| \le |\qdexp[r] - d(\x, \y)^2| + |d(\x, \y)^2 - d(x, y)^2|$, we can quantify the relative error of the computed distance in low precision with respect to the original distance:
\begin{equation}\label{eq:quantis-rel-bound-expand}
    \frac{|\qdexp[r] - d(x, y)^2|}{d(x, y)^2}   \le \frac{E^{\ell}(\x,\y) + C(x,y)}{d(x, y)^2} 
    \le \frac{2}{\wtrho -2}.
\end{equation}

\begin{revision}
A larger $\widetilde\rho$ tightens the relative-error bound~\eqref{eq:quantis-rel-bound-expand}, potentially requiring more recomputation; the bound is $2/3$ at $\widetilde\rho=5$. The thresholds used in the experiments are specified in Section~\ref{sec:implemented-threshold}.
\end{revision}

The discussion motivates the mixed-precision fallback rule
\begin{equation} \label{eq:fallback-rule-mp}
\widehat d(\x, \y)^2 = \begin{cases}
\qdexp[r]\; \text{in precision}\; \ulow, & \text{if}\; \qdexp[r] > \wtrho E^{\ell}(\x,\y),\\ 
\ddir[r]\; \text{in precision}\; u_h, & \text{otherwise},
\end{cases}
\end{equation}
where $\widehat d(\x, \y)^2$ denotes the selected approximation to the squared distance between $x$ and $y$, with input quantization $\x = \fllow(x)$ and $\y = \fllow(y)$. 
The fallback computation is performed in high-precision to retain access to the original working-precision data.
The resulting mixed-precision distance computing scheme is given as Algorithm~\ref{alg:mp-dist}.

\newcommand{\funfont}[1]{\textsc{#1}}
\newcommand{\fundist}{\ensuremath{\funfont{dist}}}

\begin{revision}
\begin{algorithm2e}[t]

\caption{Mixed-precision squared distance computation under the same-precision operation model.}
\label{alg:mp-dist}
\KwIn{$x,y$ in working precision $u$; $\rho>2$; $0<u_h\le u\le u_\ell\le1/8$; $r\ge2$; $(r+2)u_\ell<1$.}
\KwOut{Approximation to $\norm{x-y}^2$, returned in precision $u$.}
\Function{\fundist\textup{(}$x,y,\rho,u_\ell,u_h$\textup{)}}{
 $\x\gets\fllow(x)$, $\y\gets\fllow(y)$\;
 $d_{xx}\gets\fllow(\x^T\x)$, $d_{yy}\gets\fllow(\y^T\y)$,
 $d_{xy}\gets\fllow(\x^T\y)$\;
 $d\gets\max\{\fllow(d_{xx}-2d_{xy}+d_{yy}),0\}$\;
 Obtain $\overline S\ge\norm{\x}^2+\norm{\y}^2$, including rounding in this bound\; 
 \label{alg:mp-dist-line:S}
 Obtain $E\ge2\gamma^\ell_{r+2}\overline S$ and $T\ge\rho E$, without rounding either bound downward\;
 \label{alg:mp-dist-line:E}
 \If{$d\le T$ or the expanded distance computation is nonfinite}{
  $d\gets\flhigh((x-y)^T(x-y))$\;
 }
 \Return $\fl(d)$\;
}
\end{algorithm2e}
In Algorithm~\ref{alg:mp-dist}, taking the maximum with zero prevents a negative squared distance caused by rounding errors. The upper bounds in Lines~\ref{alg:mp-dist-line:S}--\ref{alg:mp-dist-line:E} ensure that rounding in the norm sum, error bound and threshold does not weaken the acceptance test~\eqref{eq:quantis-reliability-test}.  The factor $2$ follows from~\eqref{eq:expanded-abs-bound}, applied to the quantized inputs.
For a high-precision direct fallback returned in working precision, the standard model gives
\begin{equation*}
\frac{|\fl(\ddir[r])-d(x,y)^2|}{d(x,y)^2}
\le (1+u)(1+\gamma^{h}_{r+2})-1,
\end{equation*}
so output conversion can limit the accuracy even when $u_h\ll u$.
\end{revision}

\begin{revision}
\subsection{Modeling accumulation and intermediate storage}\label{sec:accum-storage}
Here we separately consider precisions for accumulation and intermediate storage. The accumulation format alone does not specify the summation order and rounding behavior needed for error analysis, as studied in the Tensor Core arithmetic~\cite{fasi2021tensor}. Therefore, we distinguish input conversion, inner-product arithmetic and intermediate storage, using a separate error parameter for the computation of inner product. Let $\x=\fllow(x)$ and $\y=\fllow(y)$, and suppose an inner-product kernel, before output conversion, satisfies
\begin{equation}\label{eq:inner-product-error-assumption}
|\fl(a^Tb)-a^Tb|\le\varepsilon_a |a|^T|b|
\end{equation}
for each of $(a,b)=(\x,\x),(\x,\y),(\y,\y)$. The parameter $\varepsilon_a$ bounds arithmetic error in the inner product, independently of input quantization. Sequential accumulation in precision $u_a$ satisfies $\varepsilon_a=\gamma_r^{(a)}$ under~\eqref{eq:fpmodel}, provided $ru_a<1$. Note that applying this value to a library kernel requires the same inner-product error bound~\eqref{eq:inner-product-error-assumption} to hold. 

Let $\gamma_2^{(w)}=2u/(1-2u)$. Suppose the three results are stored in precision $u_s$, and the expanded expression is then formed in working precision $u$. 
For the quantized inputs, the stored inner products satisfy the error bounds in the proof of Theorem~\ref{thm:dist-exp}, with $\gamma_r$ replaced by $(1+\varepsilon_a)(1+u_s)-1$. Repeating that proof with the same two final rounding factors gives
\begin{equation*}
|\widehat D-d(\x,\y)^2|
\le \bigl[(1+\varepsilon_a)(1+u_s)(1+\gamma_2^{(w)})-1\bigr]
\cdot \bigl(\norm{\x}^2+2|\x|^T|\y|+\norm{\y}^2\bigr).
\end{equation*}
Using $2|\x|^T|\y|\le\norm{\x}^2+\norm{\y}^2$ and adding the input-conversion bound $C(x,y)$ from~\eqref{eq:quantis-convert-error} by the triangle inequality gives
\begin{equation}\label{eq:accum-storage-bound}
|\widehat D-d(x,y)^2|
\le C(x,y)+2\beta(\norm{\x}^2+\norm{\y}^2),
\end{equation}
where
\begin{equation}\label{eq:accum-storage-beta}
\beta=(1+\alpha)(1+\gamma_2^{(w)})-1, \qquad 
\alpha=(1+\varepsilon_a)(1+u_s)-1.
\end{equation}
The bound also holds after mapping a negative computed value to zero, since it cannot increase its error relative to a nonnegative exact squared distance. If an additional output conversion with unit roundoff $u_o$ is applied, a bound $B$ before conversion becomes $(1+u_o)B+u_o d(x,y)^2$.

The bound~\eqref{eq:accum-storage-bound} separates input-conversion error $C(x,y)$ from arithmetic and storage errors. For sequential accumulation of  an inner product of length $r$ with $r u_a <1$, using FP32 accumulation instead of FP16/BF16 reduces the bound $\gamma_r^{(a)}$, while storing the computed squared norms and cross inner product in FP32 reduces $u_s$ in~\eqref{eq:accum-storage-beta}; neither change removes the input conversion error $C(x,y)$. Computing the norm terms requires $O((m+n)r)$ operations, compared with $O(mnr)$ for the cross product, and their accumulation already uses FP32 in our implementation.

Given a reliably computed upper bound $B\ge|\widehat D-d(x,y)^2|$, the test $\widehat D>\rho B$ with $\rho>1$ implies
\begin{equation*}
\frac{|\widehat D-d(x,y)^2|}{d(x,y)^2}\le\frac1{\rho-1}.
\end{equation*}
At $\rho=5$, the bound is $1/4$. In the context of k-means clustering, distances satisfying this bound can give different nearest-center assignments when competing distances are close; Section~\ref{sect:center-update} analyzes the center update with cluster memberships held fixed.

For cuBLAS, the bound~\eqref{eq:accum-storage-bound} requires assumption~\eqref{eq:inner-product-error-assumption}; FP32 compute alone does not imply that this assumption holds for the library kernel, since a bound on the accumulation error also depends on how the products are summed. The computed acceptance threshold $T$ must also satisfy $T\ge\rho B$. Supplementary Section~S1.1 derives an upper bound and its evaluation with directed rounding for sequential FP32 accumulation, including input-conversion residuals and gradual underflow. Section~\ref{sec:new-ablation} compares intermediate storage and threshold choices using cuBLAS, subject to the assumption on the error in the inner product.
\end{revision}

\subsection{Cluster center update}
\label{sect:center-update}
After the selection of the $k$ initial centers, the k-means method (Algorithm~\ref{alg:kmeans}) alternates between assigning each data point to its closest center and updating each center by the mean of the points assigned to it. 
For a nonempty cluster $S_i$, with $|S_i| = m_i$, the exact center update and a componentwise error bound of it are
\begin{equation}\label{eq:center}
    \mu_i=\frac{1}{m_i}\sum_{p\in S_i}p,\quad 
    \bar a_i: = \frac{1}{m_i}\sum_{p\in S_i} |p| \ge \mu_i.
\end{equation}
The update is structurally simple, but its accuracy is governed by the accuracy of the coordinate summations used to form the cluster mean. 

\begin{lemma}\label{lemma:mufloat}
Assume that the sum in~\eqref{eq:center} is computed by recursive summation in precision $u$ such that \(m_i u<1\). Then the computed center satisfies
\begin{equation*}
    \widehat{\mu}_i=\mu_i+\Delta\mu_i,
    \quad 
    |\Delta\mu_i| \le \gamma_{m_i} \bar a_i.
\end{equation*}
Consequently,
\begin{equation*}
    \norm{\Delta\mu_i} \le \gamma_{m_i}
    \norm{\bar a_i} \le \gamma_{m_i}
    \Big(
        \frac{1}{m_i}\sum_{j=1}^{m_i}\norm{p_j}^2
    \Big)^{1/2}.
\end{equation*}
\begin{proof}
The proof is straightforward from the standard error analysis~\cite[sect.~4]{high:ASNA2} and is provided in Supplementary Section S1.2.
\end{proof}
\end{lemma}

The result on center perturbation can be translated into one that concerns perturbation of the cluster energy.

\begin{theorem}
\label{thm:center-energy-perturb}
Under the assumptions of Lemma~\ref{lemma:mufloat}, the energy perturbation incurred by a rounded center \(\Delta\phi_i := \phi(\widehat\mu_i,S_i)-\phi(\mu_i,S_i)\) satisfies
\begin{equation}\label{eq:energy-perturb-bound}
    \Delta\phi_i
    \le
    \gamma_{m_i}^2
    \left(
        \phi(\mu_i,S_i)+m_i\norm{\mu_i}^2
    \right).
\end{equation}

\begin{proof}
Applying Lemma~\ref{eq:converge} with
\(p=\widehat\mu_i=\mu_i+\Delta\mu_i\) gives the identity
\begin{equation*}
    \phi(\widehat\mu_i,S_i)-\phi(\mu_i,S_i)
    =
    m_i\norm{\widehat\mu_i-\mu_i}^2
    =
    m_i\norm{\Delta\mu_i}^2.
\end{equation*}
From Lemma~\ref{lemma:mufloat}, we have
\begin{equation*}
    m_i\norm{\Delta\mu_i}^2 \le m_i \gamma_{m_i}^2 \norm{\bar a_i}^2
    \le \gamma_{m_i}^2
    \sum_{p\in S_i}\norm{p}^2.
\end{equation*}
Since $\mu_i$ is the mean center of $S_i$, we have $\sum_{p\in S_i} (p - \mu_i) = \sum_{p\in S_i}p - m_i\mu_i = 0$, and so
\begin{equation*}
    \sum_{p\in S_i}\norm{p}^2
    =
    \sum_{p\in S_i}\norm{p-\mu_i}^2
    +
    m_i\norm{\mu_i}^2
    =
    \phi(\mu_i,S_i)+m_i\norm{\mu_i}^2,
\end{equation*}
which proves~\eqref{eq:energy-perturb-bound}.
\end{proof}
\end{theorem}

The bound~\eqref{eq:energy-perturb-bound} tells us that the increase in the cluster energy caused by using the rounded center $\widehat\mu_i$ depends on both the within-cluster variation and the squared norm of the exact center $\mu_i$.  

As mentioned in the previous section, the convergence of the k-means algorithm relies on its property as a local improvement heuristic, that is, 
reassigning the data points to their closest center and then updating the cluster centers by the mean can only decrease the \sse\ of~\eqref{eq:sse}.
We now give a descent-preservation result for one center update on a fixed
cluster.  The local statement shows when the \emph{rounded} center update still decreases the cluster contribution to the \sse.

\begin{theorem}\label{theorem:center-update-descent}
Let \(\widehat c\) denote the previous computed center
and let 
\begin{equation}\label{eq:generic-center-error-bound}
   \widehat\mu=\mu+\Delta\mu, \quad  |\Delta\mu|\le \eta,
\end{equation}
denote the newly computed center. 
Then we have $\phi(\widehat\mu,S)<\phi(\widehat c,S)$ if 
\begin{equation}\label{eq:generic-descent-cond}
    \norm{\widehat c-\widehat\mu}^2 > 2|\widehat c-\widehat\mu|^T\eta.
\end{equation}

\begin{proof}
Using Lemma~\ref{eq:converge} twice, first with
\(p=\widehat c\) and then with \(p=\widehat\mu\), and then subtracting yields
\begin{equation*}
    \phi(\widehat c,S)-\phi(\widehat\mu,S) = 
    m\norm{\widehat c-\mu}^2 - m\norm{\Delta\mu}^2.
\end{equation*}
Since \(\widehat c - \mu = \widehat c-\widehat\mu + \Delta\mu\), this expression becomes
\begin{equation*}
    \phi(\widehat c,S)-\phi(\widehat\mu,S)
     = m\norm{\widehat c-\widehat\mu+\Delta\mu}^2 - m\norm{\Delta\mu}^2  
     =
    m\norm{\widehat c-\widehat\mu}^2+2m (\widehat c-\widehat\mu)^T\Delta\mu.
\end{equation*}
Using the componentwise bound~\eqref{eq:generic-center-error-bound}, we get the bound
\begin{equation*}
    (\widehat c-\widehat\mu)^T\Delta\mu
    \ge
    -|\widehat c-\widehat\mu|^T|\Delta\mu|
    \ge
    -|\widehat c-\widehat\mu|^T\eta,
\end{equation*}
and hence
\begin{equation*}
    \phi(\widehat c,S)-\phi(\widehat\mu,S) \ge
    m\left(
        \norm{\widehat c-\widehat\mu}^2-2|\widehat c-\widehat\mu|^T\eta
    \right).
\end{equation*}
Thus,~\eqref{eq:generic-descent-cond} is a sufficient condition for $\phi(\widehat c,S)-\phi(\widehat\mu,S)>0$.
\end{proof}
\end{theorem}

If the center is computed by the recursive summation of Lemma~\ref{lemma:mufloat}, such that $|\widehat\mu - \mu| \le \gamma_{m} \bar a$ for some $\bar a$, then Theorem~\ref{theorem:center-update-descent} guarantees a decrease in the energy function for the fixed cluster whenever
$\norm{\widehat c-\widehat\mu}^2 >
    2\gamma_m |\widehat c-\widehat\mu|^T\bar a$.
Furthermore, if $|\widehat c-\widehat\mu|^T\bar a \ne 0$, then this sufficient condition can be written as $ mu/ (1-mu) < \norm{\widehat c-\widehat\mu}^2/ (2|\widehat c-\widehat\mu|^T\bar a)$, or equivalently by
$$    
    u <
    \frac{\norm{\widehat c-\widehat\mu}^2}
    {m(\norm{\widehat c-\widehat\mu}^2+2|\widehat c-\widehat\mu|^T\bar a)}.
$$

Theorem~\ref{theorem:center-update-descent} assumes fixed cluster membership. It does not guarantee that a 
complete Lloyd iteration decreases the \sse, since rounding can change the nearest-center assignment when distances are nearly equal.  Summing the same argument over all clusters gives a corresponding sufficient condition for descent of the total \sse\ after one center-update step, given as Corollary~\ref{cor:total-sse-decrease}.

\begin{corollary}
\label{cor:total-sse-decrease}
Let \(S_1,\ldots,S_k\) be the current nonempty clusters, and let
\(\widehat c_i\) and \(\widehat\mu_i\) denote the previous and newly computed
centers for \(S_i\), respectively. Under the assumptions of Lemma~\ref{lemma:mufloat},
\begin{equation*}
    \sum_{i=1}^k
    \left[
        \phi(\widehat c_i,S_i)-\phi(\widehat\mu_i,S_i)
    \right] 
    \ge
    \sum_{i=1}^k
    m_i
    \left(
        \norm{\widehat c_i-\widehat\mu_i}^2
        -
        2\gamma_{m_i}|\widehat c_i-\widehat\mu_i|^T\bar a_i
    \right).
\end{equation*}
\end{corollary}

The vector $|\widehat c-\widehat\mu|$ measures the center movement, which is expected to decrease as the iterations proceed. Therefore, the update may require either higher precision or a more accurate summation scheme to preserve the descent property at the level of the computed centers.
For the latter, one possibility is to employ pairwise summation~\cite[sect.~4.1]{high:ASNA2}, and then the factor \(\gamma_m\) in the bounds above can be replaced by a factor depending on the summation depth, typically of order \(\gamma_{\lceil \log_2 m\rceil}\). 

Finally, we note that the analysis is not independent of preprocessing, which changes the data distribution. Indeed, the error bounds for the center update depend on quantities such as $\bar a_i = \sum_{p\in S_i}|p|/m_i$ and $\sum_{p\in S_i}\norm{p}^2$, which may change when the data are rescaled, centered, or normalized.

\section{Mixed-precision k-means}\label{sect:mp-kmeans}

\begin{algorithm2e}[!t]
	
\caption{Mixed-precision Lloyd's algorithm.}
	\label{alg:mp-kmeans}
	\KwIn{Data set $P=[p_i]_{i=1}^{n}$ in working precision $u$, safety factor
		$\rho>2$, precisions $\ulow$ and $\uhigh$, where $0<\uhigh\le u\le\ulow$, the number of clusters $k$,
		convergence tolerance $\tau>0$, and maximum number of iterations $t_{\max}$.}
	\KwOut{Centers $\{c_1,\ldots,c_k\}$ stored in precision $u$ and cluster assignments.}
		Normalize dataset $P$ in precision $u$\;
		Select initial centers $\{c_1^{(0)},\ldots,c_k^{(0)}\}$\;
		\For{$t \gets 1,\ldots,t_{\max}$}{
			\ForEach{$p_i \in P$}{
				\ForEach{nonempty cluster $S_j^{(t)}$}{
					$\widehat d_{ij} \gets \fundist\textup{(}p_i,c_j^{(t-1)},\rho,\ulow,\uhigh\textup{)}$
					\tcp*{Call Algorithm~\ref{alg:mp-dist}}
				}
				Assign $p_i$ to cluster $S_j^{(t)}$, where $j=\argmin_{s} \widehat d_{is}$\;
			}
			\ForEach{nonempty cluster $S_j^{(t)}$}{
				$c_j^{(t)}	\gets \fl\left(\frac{1}{|S_j^{(t)}|} \sum_{p_i\in S_j^{(t)}} p_i\right)$ in precision $u$\;
			}
			Reinitialize each empty cluster by a randomly selected data point in precision $u$\;
			\If{$\max_j \norm{c_j^{(t)}-c_j^{(t-1)}} < \tau$}{
				\textbf{break}\;
			}
		}
		Recompute all distances by Algorithm~\ref{alg:mp-dist} and assign each
		$p_i$ to closest center\;
		\Return final centers $\{c_1,\ldots,c_k\}$ and cluster assignments
\end{algorithm2e}

Based on our discussion in the previous section, we present our mixed-precision k-means algorithm in Algorithm~\ref{alg:mp-kmeans}.

The distance computations used for cluster assignment are performed in mixed precision using Algorithm~\ref{alg:mp-dist}.  In contrast, the centroid updates are carried out in the working precision throughout.  This choice is motivated by the analysis in Section~\ref{sect:center-update}, which shows that sufficient precision in the centroid update helps preserve descent for fixed cluster memberships, especially when center displacements become small.

Reduced-precision center updates may be viable early in the iteration, but we do not pursue this option here because a reliable switching criterion is not yet available.

\section{Numerical experiments}\label{sec:numer-exper}
The numerical simulations and discussions are structured as follows. Section~\ref{sect:numer-experi-dist-comput} examines the accuracy and runtime of the squared-distance kernels, including the effects of intermediate storage and the reliability test. Section~\ref{supp:stress} studies sensitivity to near-duplicate points, and Section~\ref{sec:clustering-implementation} evaluates the effects on clustering quality and runtime using Gaussian data and image segmentation. The compute environment is based on CUDA 12.8 and \texttt{PyTorch} 2.11.0+cu128, running on a single compute node allocated with one NVIDIA A100 GPU. For Gaussian clustering, we standardize each feature using z-score normalization, i.e., by subtracting its mean and dividing by its standard deviation~\cite{han2011datamining}. The experimental code is available at \url{https://github.com/inEXASCALE/mp-kmeans}. Additional details on the experiment setup and further complementary test results are provided in Supplementary Sections~S2--S4.

\subsection{Euclidean distance computation}\label{sect:numer-experi-dist-comput}
First, we compare the squared Euclidean distances computed by~\eqref{eq:dist-expanded} with inputs in BF16, FP16, FP32, and FP64 (see Table~\ref{table:float-arith-parameter}), respectively. The results for the direct formula~\eqref{eq:dist-direct} in uniform FP32 and FP64 (without fallback) are also presented for comparison, and the latter is used as the reference value for accuracy.
For the uniform FP32 and FP64 kernels (without fallback) based on~\eqref{eq:dist-expanded}, we use the native pairwise distance interface \texttt{cdist} in \texttt{PyTorch}. \texttt{PyTorch}'s \texttt{cdist} operator is implemented through ATen, \texttt{PyTorch}'s C++ tensor library, and dispatches to the corresponding CUDA kernel when the input tensors reside on a CUDA device.
The expanded FP32/FP64 call is \texttt{torch.cdist(P,C,p=2.0)**2}, including the squaring, where its matrix-multiplication formulation depends on matrix sizes. The direct FP64 reference uses the original FP32 data cast exactly to FP64 and is retained without output rounding.

The uniform BF16 and FP16 kernels (without fallback) are based on our own CUDA implementations. For the no-fallback FP16/BF16 expanded methods, the format label specifies the input coordinates, stored squared norms and GEMM results, and returned distances.
These CUDA kernels are exposed to Python via \texttt{pybind11}\footnote{\url{https://github.com/pybind/pybind11}} bindings and are invoked from Python for evaluation.

\begin{revision}
For a method denoted $A\rightarrow B$, $A$ specifies the format used for the expanded distance calculation and $B$ the arithmetic used for direct recomputation after a failed test. These interfaces accept point matrices and return squared-distance matrices stored in FP32, including $A\rightarrow$FP64; Table~\ref{tab:arithmetic-paths} specifies the precision for accumulation and final additions and subtraction in the expanded formula. The mixed-precision implementation first evaluates the expanded distances, tests their reliability, and recomputes rejected entries using the direct formula.
The fallback fraction is
\begin{equation}\label{eq:eta}
\eta=\frac{\text{number of recomputed distance entries}}
{\text{total number of distance entries}},
\end{equation} 
where the computed percentages in the distance tables are rounded, so a displayed ratio $0.00\%$ can represent a small \emph{nonzero} recomputation fraction.

\subsubsection{GPU arithmetic and storage paths}\label{sec:gpu-arithmetic}
Table~\ref{tab:arithmetic-paths} summarizes the precisions used at each stage of the distance routines. 
The mixed FP16/BF16 distance routines call \texttt{cublasGemmEx} with input and output types \texttt{CUDA\_\allowbreak R\_\allowbreak 16F} or \texttt{CUDA\_\allowbreak R\_\allowbreak 16BF}, respectively, compute type \texttt{CUBLAS\_\allowbreak COMPUTE\_\allowbreak 32F}, and algorithm selector \texttt{CUBLAS\_\allowbreak GEMM\_\allowbreak DEFAULT\_\allowbreak TENSOR\_\allowbreak OP}. These calls specify FP32 compute and low-precision GEMM output, while allowing cuBLAS to select Tensor Core kernels. The precise library-selected device instructions are not established by a source-level call alone\footnote{\url{https://docs.nvidia.com/cuda/archive/12.8.0/cublas/index.html}}. The FP32 compute setting does not specify the accumulation order. We therefore treat the inner-product error bound in Section~\ref{sec:accum-storage} as an assumption for the cuBLAS kernel.

The norm kernels convert each FP16/BF16 input element exactly to FP32, accumulate its square using FP32 \texttt{fmaf}, and round the resulting norm back to the input format. The stored norm and GEMM values are widened to FP32 for the expanded-form additions, nonnegative mapping, and threshold comparison. FP32 or FP64 direct fallback uses the FP32 tensors passed to the distance interface, with the differences and accumulation evaluated in the selected correction precision. FP64 accumulated results are cast to FP32 when written to the returned distance matrix.
\end{revision}
\begin{table}[t]
\centering
\caption{Precisions used by the distance routines. $L$ is FP16 or BF16; FP32 request denotes the cuBLAS compute type.}
\label{tab:arithmetic-paths}
\small\renewcommand{\arraystretch}{1.12}\resizebox{0.73\linewidth}{!}{
\begin{tabular}{lccc}
\toprule
Stage & $L$ (no fallback) & $L\rightarrow$FP32 & $L\rightarrow$FP64\\
\midrule
Interface input & $L$ & FP32 & FP32\\
GEMM input / stored result & $L/L$ & $L/L$ & $L/L$\\
GEMM interface & \texttt{torch::mm} & \texttt{GemmEx} & \texttt{GemmEx}\\
GEMM compute specification & PyTorch policy & FP32 request & FP32 request\\
Norm accumulation / storage & FP32/$L$ & FP32/$L$ & FP32/$L$\\
Expanded-form arithmetic & FP32 & FP32 & FP32\\
Direct fallback arithmetic & -- & FP32 & FP64\\
Returned distances & $L$ & FP32 & FP32\\
\bottomrule
\end{tabular}}
\end{table}

The $L\rightarrow L$ methods recompute rejected distances on quantized inputs using an $L$ accumulator and return FP32. The FP32$\rightarrow$FP32 and FP32$\rightarrow$FP64 methods use FP32 norms, stored intermediates and \texttt{torch::mm}. PyTorch's matrix-multiplication policy and library reductions are separate from the same-precision sequential model. Section~\ref{sec:execution-costs} reports measured execution costs; see Supplementary Section~S2.1 for additional source and build details.

\subsubsection{Empirical distance threshold}\label{sec:implemented-threshold}
The distance routines use the threshold $T_{\rm impl}=\rho\,c_A(r)(\widehat p+\widehat q)$,  evaluated in FP32, where $\widehat p,\widehat q$ are the stored norm terms and $c_A(r)$ is specified in Table~\ref{tab:threshold-coefficients}. The BF16 constant $2^{-7}$ and the FP32 constant $2^{-23}$ in both correction configurations are machine epsilons, equal to twice the respective unit roundoffs. The FP16 constant $2^{-11}$ is its unit roundoff. We use the smallest of the three empirical settings $\{5,10,50\}$ reported in Supplementary Section~S4, $\rho=5$, as the default. Increasing $\rho$ makes the implemented test more conservative.

The coefficients in the table differ from the $\gamma^\ell_{r+2}$ coefficient in the analysis, so the relative-error bound~\eqref{eq:quantis-rel-bound-expand} cannot be applied directly to this test. Section~\ref{sec:accum-storage} provides an extended error bound that separates input conversion, accumulation and intermediate storage, under the inner product error assumption~\eqref{eq:inner-product-error-assumption}.
In particular, the relative-error bound $1/4$ at $\rho=5$ applies to distances accepted using a reliably computed upper bound as in Section~\ref{sec:accum-storage}, not automatically to the empirical coefficients in Table~\ref{tab:threshold-coefficients}.

\begin{revision}
At $r=128$, the BF16 coefficient gives $\rho\,c_A(r)=10.1563$. By comparison, for exact quantized inputs the squared distance is at most $2(\|\x\|^2+\|\y\|^2)$. This large threshold explains the strong tendency towards BF16 fallback in the following experiments. Higher-precision accumulation alone cannot reduce the empirical threshold if the coefficient is left unchanged. The effects of changing the threshold and storing the computed squared norms and cross inner products in FP32 are evaluated in Section~\ref{sec:new-ablation}.

\begin{table}[t]
\centering
\caption{Coefficients $c_A(r)$ used in the implemented distance thresholds for different arithmetic paths.}
\label{tab:threshold-coefficients}
\small\renewcommand{\arraystretch}{1.12}
\resizebox{0.43\linewidth}{!}{
\begin{tabular}{cr}
\toprule
Default/correction path & $c_A(r)\qquad$\\
\midrule
FP16 default & $(2r+4)\cdot 2^{-11}$\\
BF16 default & $(2r+4)\cdot 2^{-7\phantom{1}}$\\
FP32$\rightarrow$FP32 & $(2r+4)\cdot 2^{-23}$\\
FP32$\rightarrow$FP64 & $2r\cdot 2^{-23}$\\
\bottomrule
\end{tabular}}
\end{table}

\end{revision}

\subsubsection{Distance on far and near points}

\begin{table}[!t]
\centering\small
\setlength{\tabcolsep}{3pt}
\caption{Far/near distance benchmark with 25 million entries. Input/output denotes the interface format (equal for these rows). In the Method column, Expanded and Direct denote the expanded and direct distance formulas, respectively. Speedups are relative to D05, which evaluates the expanded formula in FP64, and errors are computed by comparing with the unrounded FP64 reference in D06.}
\label{tab:per_random_near}\renewcommand{\arraystretch}{1.12}\resizebox{0.73\linewidth}{!}{
\begin{tabular}{lllr r l r}
\toprule
\multicolumn{7}{c}{Far points}\\
\midrule
ID & Input/output & Method & Time (ms) & Speedup & Max. rel. error & $\eta$\\
\midrule
D01 & FP16 & FP16 & 0.53 & 5.17$\times$ & $1.218 \times 10^{-3}$ & -\\
D02 & BF16 & BF16 & 0.58 & 4.72$\times$ & $9.106 \times 10^{-3}$ & -\\
D03 & FP32 & FP32 Expanded & 1.83 & 1.50$\times$ & $4.942 \times 10^{-7}$ & -\\
D04 & FP32 & FP32 Direct & 30.68 & 0.09$\times$ & $9.599 \times 10^{-7}$ & -\\
D05 & FP64 & FP64 Expanded & 2.74 & 1.00$\times$ & $5.960 \times 10^{-8}$ & -\\
D06 & FP64 & FP64 Direct & 57.76 & 0.05$\times$ & - & -\\
D07 & FP32 & FP16$\rightarrow$FP16 & 0.67 & 4.09$\times$ & $5.672 \times 10^{-3}$ & 0.00\%\\
D08 & FP32 & BF16$\rightarrow$BF16 & 33.11 & 0.08$\times$ & $6.540 \times 10^{-2}$ & 100.00\%\\
D09 & FP32 & FP32$\rightarrow$FP32 & 1.47 & 1.86$\times$ & $8.554 \times 10^{-7}$ & 0.00\%\\
D10 & FP32 & FP16$\rightarrow$FP32 & 0.67 & 4.09$\times$ & $8.587 \times 10^{-4}$ & 0.00\%\\
D11 & FP32 & FP16$\rightarrow$FP64 & 0.68 & 4.03$\times$ & $8.587 \times 10^{-4}$ & 0.00\%\\
D12 & FP32 & BF16$\rightarrow$FP32 & 33.12 & 0.08$\times$ & $9.599 \times 10^{-7}$ & 100.00\%\\
D13 & FP32 & BF16$\rightarrow$FP64 & 33.04 & 0.08$\times$ & $5.960 \times 10^{-8}$ & 100.00\%\\
D14 & FP32 & FP32$\rightarrow$FP64 & 1.49 & 1.84$\times$ & $8.554 \times 10^{-7}$ & 0.00\%\\
\midrule
\multicolumn{7}{c}{Near points}\\
\midrule
ID & Input/output & Method & Time (ms) & Speedup & Max. rel. error & $\eta$\\
\midrule
D01 & FP16 & FP16 & 0.50 & 5.42$\times$ & $\mathbf{1.908 \times 10^{9}}$ & -\\
D02 & BF16 & BF16 & 0.53 & 5.11$\times$ & $\mathbf{3.480 \times 10^{9}}$ & -\\
D03 & FP32 & FP32 Expanded & 1.80 & 1.51$\times$ & $\mathbf{1.692 \times 10^{6}}$ & -\\
D04 & FP32 & FP32 Direct & 30.67 & 0.09$\times$ & $9.328 \times 10^{-7}$ & -\\
D05 & FP64 & FP64 Expanded & 2.71 & 1.00$\times$ & $2.709 \times 10^{-3}$ & -\\
D06 & FP64 & FP64 Direct & 57.76 & 0.05$\times$ & - & -\\
D07 & FP32 & FP16$\rightarrow$FP16 & 0.68 & 3.99$\times$ & $\mathbf{3.716 \times 10^{4}}$ & 0.02\%\\
D08 & FP32 & BF16$\rightarrow$BF16 & 33.10 & 0.08$\times$ & $\mathbf{2.349 \times 10^{6}}$ & 100.00\%\\
D09 & FP32 & FP32$\rightarrow$FP32 & 1.47 & 1.84$\times$ & $8.938 \times 10^{-7}$ & 0.02\%\\
D10 & FP32 & FP16$\rightarrow$FP32 & 0.67 & 4.04$\times$ & $8.524 \times 10^{-4}$ & 0.02\%\\
D11 & FP32 & FP16$\rightarrow$FP64 & 0.68 & 3.99$\times$ & $8.524 \times 10^{-4}$ & 0.02\%\\
D12 & FP32 & BF16$\rightarrow$FP32 & 33.13 & 0.08$\times$ & $9.328 \times 10^{-7}$ & 100.00\%\\
D13 & FP32 & BF16$\rightarrow$FP64 & 32.99 & 0.08$\times$ & $5.960 \times 10^{-8}$ & 100.00\%\\
D14 & FP32 & FP32$\rightarrow$FP64 & 1.49 & 1.82$\times$ & $8.938 \times 10^{-7}$ & 0.02\%\\
\bottomrule
\end{tabular}}
\end{table}

\begin{revision}
We use two feature matrices of size $5,000\times128$, yielding 25 million distance entries. The far scenario uses independent standard Gaussian FP32 vectors. In the near scenario, $C_i=P_i+\varepsilon_i$ with noise of order $10^{-6}$, generated in FP32. This creates near-duplicate \emph{matched} pairs, while most off-diagonal pairs remain far apart. All formats are obtained from the same original FP32 data. The generated Gaussian vectors are used directly, without an additional z-score pass.

The distance benchmark uses four runs: the first is a warm-up excluded from timing statistics, and the remaining three are averaged. Warm-up removes first-use costs such as context or library initialization and lazy loading from the measured steady-state calls. Additional details on the timing are given in Supplementary Section~S2.2.
The reported maximum relative error of each method is
$$
\max_{ij}\frac{|D^{\rm eval}_{ij}-D^{\rm ref}_{ij}|}{|D^{\rm ref}_{ij}|+10^{-30}},
$$
where $D^{\rm eval}$ is the evaluated squared-distance matrix $D$ defined in~\eqref{eq:dist-blas3}, and $D^{\rm ref}$ is computed using the direct formula~\eqref{eq:dist-direct} in FP64 on the original FP32 inputs. Evaluated distances are converted to FP32 for the error comparison. The error ratio is computed in higher precision and is thus not limited by the FP16 range. Table~\ref{tab:per_random_near} presents the results, with speedups relative to D05, which evaluates the expanded formula in FP64.

On near data, a small reference squared distance can produce a very large relative error. The expanded formula can still suffer larger cancellation errors on the closest pairs, even in FP64. The conversion to FP32 explains the error level near $5.96\times10^{-8}$ for D05 on far data and for D13 on both far and near data. Additionally, near-pair distances recomputed using the direct formula in FP64 and returned in FP32 also exhibit errors at this level. 
For D08 and D12, which both recompute every pair, the far-point errors are $6.540\times10^{-2}$ and $9.599\times10^{-7}$, respectively; the former uses quantized BF16 inputs and BF16 accumulation, whereas the latter uses original FP32 inputs and FP32 arithmetic. D13 also recomputes every pair from the original FP32 inputs, but in FP64, whereas D14 retains almost all FP32 default-path results. This explains why BF16$\rightarrow$FP64 can be more accurate than FP32$\rightarrow$FP64, not because BF16 arithmetic is more accurate, but because it triggers substantially more high-precision recomputations.  For FP16$\rightarrow$FP32/FP64, correction improves the near entries while the maximum error remains about $8.6\times10^{-4}$ because most far-point distances retain the initial expanded evaluation in FP16. Note that methods without recomputation and those with same-precision recomputation use different output formats and library routines, so equal displayed fallback rates do not imply identical arithmetic being used in these cases.

When most entries are rejected, the initial evaluation and testing add to the cost of direct recomputation. Section~\ref{sec:execution-costs} reports the time spent on these stages.  For the total runtime of the distance routines, Table~\ref{tab:per_random_near} compares GPU implementations on the same device. D10/D11 take about $0.67$--$0.68$ ms, compared with $1.83$ ms for D03, which evaluates the expanded formula in FP32, and $2.74$ ms for D05, which evaluates it in FP64, on far points: approximately $2.7\times$ and $4.0\times$ speedups, respectively. Their overall maximum error is of order $10^{-3}$. The 100\% fallback BF16 variants take about $33$ ms and are roughly $12\times$ slower than D05. Note that these comparisons include implementation and output-format differences and do not isolate the effect of arithmetic precision alone.

The distance implementation used in Table~\ref{tab:per_random_near} first evaluates all expanded distances, records rejected entries in an index list using an atomic counter, and recomputes them with a separate direct kernel. The recomputation kernel is launched with enough threads to cover all $mn$ pairs; threads whose indices lie beyond the rejection list return immediately. At high fallback rates, the default computation, list traffic and synchronization are paid in addition to direct recomputation. For FP32 and FP64 correction, the direct kernel reads the original FP32 interface tensors and writes FP32 distances in both cases. Consequently, selecting a BF16 default path does not halve the correction kernel's input/output traffic. The runtime can be decomposed as
\begin{equation*}
T\simeq T_{\rm cast}+T_{\rm norm}+T_{\rm GEMM}
+T_{\rm test/list}+T_{\rm direct}(\eta)+T_{\rm sync}.
\end{equation*}
where $T_{\rm cast}$ is the time for precision conversion, $T_{\rm norm}$ for computing squared norms, and $T_{\rm GEMM}$ for computing the cross inner products, $T_{\rm test/list}$ for forming the expanded distances, testing their reliability and constructing the index list of rejected entries, $T_{\rm direct}(\eta)$ for the direct-recomputation time when a fraction $\eta$ of the distances is rejected, including the overhead for kernel launch, and $T_{\rm sync}$ for synchronization overhead.  

The cross product requires $F\simeq2mnr$ operations and can reuse input blocks across many entries. Evaluating error bounds and identifying rejected entries instead operate on individual distances. For a stage with operation count $F$ and data movement $Q$, a runtime lower bound is $T\ge\max(F/P,Q/W)$, where $P$ and $W$ are the arithmetic rate and memory bandwidth attainable for that stage on the measured GPU allocation.

For $m=n=5,000$ and $r=128$, the two FP32 input matrices occupy $5.12$ MB while an FP32 distance output occupies $100$ MB. Halving only the input storage would therefore reduce this idealized total by about $2.4\%$. Actual kernels also access stored cross products and fallback lists and may reread inputs. Direct fallback on $\eta mn$ pairs uses approximately $3\eta mnr$ floating-point operations when one subtraction and one fused multiply-add are counted as three operations per coordinate.  This operation count excludes precision conversion, index list processing and kernel launch overhead. Further, cache reuse also prevents source-level load counts from determining the  actual traffic $Q$. Supplementary Section~S3 provides measured kernel-stage times. These timings identify the most time-consuming stages, but without hardware performance counters  they do not establish whether arithmetic throughput or memory bandwidth is the main bottleneck.

\subsubsection{Effect of intermediate storage and reliability tests}\label{sec:new-ablation}
We compare eight combinations of FP16/BF16 inputs, low/FP32 intermediate storage and two reliability tests, using $n=m=5000$ (dimensions as in~\eqref{eq:dist-blas3}), $r=128$, with three different random states for data generation. 

With $\rho=5$ fixed in both tests, the thresholds are
\begin{equation*}
T_{\rm Original}=\rho\,c_A(r)(\widehat p+\widehat q), \qquad T_{\rm Extended}=\mathrm{round}_{+}(\rho B_*).
\end{equation*}
Original evaluates its threshold in FP32 using the stored squared norms $\widehat p,\widehat q$ and the coefficient $c_A(r)$ from Table~\ref{tab:threshold-coefficients}, which is unchanged when intermediate storage changes. Extended uses the error bound $B_*$ from Supplementary Section~S1.1, accounting separately for input conversion, FP32 accumulation, intermediate storage and underflow; $\mathrm{round}_{+}$ denotes upward rounding in FP64. Its application to cuBLAS uses the inner-product error assumption of Section~\ref{sec:accum-storage}. All configurations use cuBLAS FP32 compute, $\rho=5$, FP64 direct recomputation from original FP32 inputs and FP32 output. Far data contain independent Gaussian vectors; matched data include 5,000 near pairs among 25 million entries ($0.02\%$, with reference squared distance $<10^{-5}$).

Table~\ref{tab:new-distance} reports all configurations for both datasets. For each configuration in this experiment, we perform three warm-up calls followed by ten calls timed using CUDA events. The reported runtime is the median of the three timing medians for the three random states. The measured time includes precision conversion, computation of squared norms and error bounds, GEMM, identification of distances that fail the test, and their direct recomputation. The reported error is the largest per-run maximum over the three runs with different random seeds against the unrounded direct FP64 reference, and the fallback is the median recomputation fraction. Supplementary Section~S3 reports seed ranges, direct FP32/FP64 comparisons, near-pair errors, assignment differences and checks of the error bounds on accepted distances.
\par
\end{revision}
\begin{table}[htbp]
\centering\small
\caption{Effect of intermediate storage and threshold choice. All configurations use cuBLAS FP32 compute, FP64 recomputation from original FP32 inputs, FP32 output and $\rho=5$. Original uses Table~\ref{tab:threshold-coefficients}; Extended uses the error bound in Supplementary Section~S1.1. Timing includes computation of the bounds and identification of distances to recompute. Nonzero fallback rates below $0.01\%$ are shown as $<0.01\%$.}
\label{tab:new-distance}\renewcommand{\arraystretch}{1.12}\resizebox{0.72\linewidth}{!}{
\begin{tabular}{lllrrr}
\toprule
Input & Storage & Test & Time (ms) & Max rel. error & Fallback (\%)\\
\midrule
\multicolumn{6}{c}{Far pairs}\\
FP16 & FP16 & Original & 2.495 & $8.870\times10^{-4}$ & $<0.01$\\
FP16 & FP16 & Extended & 2.493 & $8.870\times10^{-4}$ & 0.00\\
FP16 & FP32 & Original & 2.219 & $2.829\times10^{-4}$ & $<0.01$\\
FP16 & FP32 & Extended & 2.214 & $2.829\times10^{-4}$ & 0.00\\
BF16 & BF16 & Original & 59.619 & $5.960\times10^{-8}$ & 100.00\\
BF16 & BF16 & Extended & 2.528 & $6.970\times10^{-3}$ & 0.00\\
BF16 & FP32 & Original & 59.342 & $5.960\times10^{-8}$ & 100.00\\
BF16 & FP32 & Extended & 2.248 & $2.332\times10^{-3}$ & 0.00\\
\midrule
\multicolumn{6}{c}{Matched pairs: $0.02\%$ near entries}\\
FP16 & FP16 & Original & 2.538 & $8.417\times10^{-4}$ & 0.02\\
FP16 & FP16 & Extended & 2.547 & $8.417\times10^{-4}$ & 0.02\\
FP16 & FP32 & Original & 2.261 & $2.798\times10^{-4}$ & 0.02\\
FP16 & FP32 & Extended & 2.266 & $2.798\times10^{-4}$ & 0.02\\
BF16 & BF16 & Original & 59.668 & $5.960\times10^{-8}$ & 100.00\\
BF16 & BF16 & Extended & 2.583 & $6.541\times10^{-3}$ & 0.02\\
BF16 & FP32 & Original & 59.397 & $5.960\times10^{-8}$ & 100.00\\
BF16 & FP32 & Extended & 2.309 & $2.226\times10^{-3}$ & 0.02\\
\bottomrule
\end{tabular}}
\end{table}

\begin{revision}
The choice of threshold has a substantial effect on BF16 recomputation. Even when accumulation and intermediate storage use FP32, the BF16 Original test still triggers $100\%$ recomputation because its threshold retains the large coefficient $\rho\,c_A(r)=10.1563$ discussed in Section~\ref{sec:implemented-threshold}. 
For BF16 inputs, the Extended test combines the input conversion bound with bounds for FP32 inner-product arithmetic and the selected intermediate storage format. With either BF16 or FP32 storage, it gives zero recomputation on far data and $0.02\%$ on near data. 
With BF16 inputs and FP32 storage on far data, time decreases from $59.342$ to $2.248$ ms, while maximum relative error increases from $5.960\times10^{-8}$ to $2.332\times10^{-3}$. The reduction in recomputation thus comes from changing the fallback test, which accepts less accurate expanded distances.

\subsubsection{Measured execution costs}\label{sec:execution-costs}
Table~\ref{tab:new-profile} reports one Nsight Systems trace per configuration on far data with seed~0. For the four Extended configurations, forming the expanded distances, evaluating the error bounds and identifying rejected entries take $1.712$--$1.718$ ms ($66.8\%$--$76.5\%$ of summed kernel time), while GEMM accounts for only $5.5\%$--$8.7\%$. The recomputation kernel still takes $0.239$--$0.242$ ms when no entry is rejected because it scans the rejection mask. For BF16$\rightarrow$BF16, direct recomputation takes $32.013$ ms ($96.8\%$). This method recomputes on quantized BF16 inputs using BF16 accumulation, unlike the FP64 recomputation used in Table~\ref{tab:new-distance}.
\par
\end{revision}

\begin{table}[htbp]
\centering\small
\caption{Kernel times (ms) within one profiled call. Test/select includes distance formation, error-bound evaluation and identification of rejected entries for Extended, and distance formation and index-list construction for BF16$\rightarrow$BF16. Other includes conversion, copying and initialization. Percentages in the text use the sum of kernel durations.}
\label{tab:new-profile}\renewcommand{\arraystretch}{1.12}\resizebox{0.73\linewidth}{!}{
\begin{tabular}{lrrrrr}
\toprule
Method / storage & GEMM & Norms & Test/select & Fallback & Other\\
\midrule
FP16 Extended / FP16 & 0.140 & 0.108 & 1.713 & 0.241 & 0.324\\
FP16 Extended / FP32 & 0.155 & 0.108 & 1.712 & 0.239 & 0.025\\
BF16 Extended / BF16 & 0.175 & 0.109 & 1.718 & 0.242 & 0.329\\
BF16 Extended / FP32 & 0.197 & 0.108 & 1.712 & 0.240 & 0.025\\
\midrule
\multicolumn{6}{l}{BF16$\rightarrow$BF16 comparison}\\
BF16$\rightarrow$BF16 & 0.175 & 0.050 & 0.830 & 32.013 & 0.014\\
\bottomrule
\end{tabular}}
\end{table}

\begin{revision}
Storing the GEMM result in FP32 avoids a roughly $0.30$ ms conversion from FP16/BF16 to FP32, although GEMM itself is slightly slower. Peak additional GPU allocation is $433{,}174{,}528$ bytes for each configuration, including FP32 copies of the quantized inputs, the stored cross-product matrix, per-point norm bounds, the per-distance error-bound array and the rejection mask. Low-precision intermediate storage therefore does not reduce the measured peak in this implementation. Supplementary Section~S3 gives the cuBLAS GEMM kernel names and the native BF16 instruction evidence for the same-format direct fallback.
\end{revision}

\subsection{Sensitivity tests on fallback}\label{supp:stress}
The fallback criterion in this subsection is the Original test described in Section~\ref{sec:implemented-threshold}.

\begin{table}
	\centering
	\caption{Near-point stress test with $\varepsilon=10^{-4}$; times are in milliseconds. The fraction refers to points in $C$, not all distance entries.  A selected fraction of points in $C$ are generated as near-duplicates of the corresponding points in $P$ according to $C_{\mathcal I}=P_{\mathcal I}+\varepsilon\mathcal{N}(0,I)$, where $\mathcal I$ indexes the selected rows, with noise $\varepsilon$, and the remaining points are sampled independently. Here $\eta$ denotes the fallback rate~\eqref{eq:eta}.}
	\label{tab:dist-streess-test}
	\resizebox{0.83\linewidth}{!}{
		\begin{tabular}{r c r l r l l }
			\toprule
			Fraction & Method & Time (ms) & Max error &  $\eta$ & Near max error & Far max error \\
			\midrule
			0\% & BF16 & 0.54 & $9.26 \times 10^{-3}$ & - & $0$ & $9.26 \times 10^{-3}$ \\
			0\% & FP16 & 0.51 & $1.15 \times 10^{-3}$ & - & $0$ & $1.15 \times 10^{-3}$ \\
			0\% & FP32  & 1.81 & $4.42 \times 10^{-7}$ & - & $0$ & $4.42 \times 10^{-7}$ \\
			0\% & FP64 & 2.71 & - & - & $0$ & - \\
			0\% & BF16$\rightarrow$BF16 & 33.12 & $6.92 \times 10^{-2}$ & 100\% & $0$ & $6.92 \times 10^{-2}$ \\
			0\% & FP16$\rightarrow$FP16 & 0.68 & $3.99 \times 10^{-3}$ & 0\% & $0$ & $3.99 \times 10^{-3}$ \\
			0\% & FP32$\rightarrow$FP32 & 1.47 & $7.70 \times 10^{-7}$ & 0\% & $0$ & $7.70 \times 10^{-7}$ \\
			0\% & BF16$\rightarrow$FP32 & 33.09 & $9.73 \times 10^{-7}$ & 100\% & $0$ & $9.73 \times 10^{-7}$ \\
			0\% & BF16$\rightarrow$FP64 & 32.99 & $5.96 \times 10^{-8}$ & 100\% & $0$ & $5.96 \times 10^{-8}$ \\
			0\% & FP16$\rightarrow$FP32 & 0.68 & $8.41 \times 10^{-4}$ & 0\% & $0$ & $8.41 \times 10^{-4}$ \\
			0\% & FP16$\rightarrow$FP64 & 0.67 & $8.41 \times 10^{-4}$ & 0\% & $0$ & $8.41 \times 10^{-4}$ \\
			0\% & FP32$\rightarrow$FP64 & 1.49 & $7.70 \times 10^{-7}$ & 0\% & $0$ & $7.70 \times 10^{-7}$ \\
			
			\midrule
			10\% & BF16 & 0.54 & $\mathbf{8.89 \times 10^{5}}$ & - & $8.89 \times 10^{5}$ & $9.14 \times 10^{-3}$ \\
			10\% & FP16 & 0.50 & $\mathbf{1.11 \times 10^{5}}$ & - & $1.11 \times 10^{5}$ & $1.17 \times 10^{-3}$ \\
			10\% & FP32  & 1.80 & $\mathbf{1.17 \times 10^{2}}$ & - & $1.17 \times 10^{2}$ & $4.58 \times 10^{-7}$ \\
			10\% & FP64 & 2.71 & - & - & - & - \\
			10\% & BF16$\rightarrow$BF16 & 33.10 & $\mathbf{2.55 \times 10^{2}}$ & 100\% & $2.55 \times 10^{2}$ & $6.85 \times 10^{-2}$ \\
			10\% & FP16$\rightarrow$FP16 & 0.67 & $\mathbf{1.03 \times 10^{1}}$ & 0\% & $1.03 \times 10^{1}$ & $4.74 \times 10^{-3}$ \\
			10\% & FP32$\rightarrow$FP32 & 1.47 & $8.44 \times 10^{-7}$ & 0\% & $6.02 \times 10^{-7}$ & $8.44 \times 10^{-7}$ \\
			10\% & BF16$\rightarrow$FP32 & 33.10 & $9.45 \times 10^{-7}$ & 100\% & $6.02 \times 10^{-7}$ & $9.45 \times 10^{-7}$ \\
			10\% & BF16$\rightarrow$FP64 & 33.00 & $5.96 \times 10^{-8}$ & 100\% & $5.95 \times 10^{-8}$ & $5.96 \times 10^{-8}$ \\
			10\% & FP16$\rightarrow$FP32 & 0.67 & $8.52 \times 10^{-4}$ & 0\% & $6.02 \times 10^{-7}$ & $8.52 \times 10^{-4}$ \\
			10\% & FP16$\rightarrow$FP64 & 0.67 & $8.52 \times 10^{-4}$ & 0\% & $5.95 \times 10^{-8}$ & $8.52 \times 10^{-4}$ \\
			10\% & FP32$\rightarrow$FP64 & 1.49 & $8.44 \times 10^{-7}$ & 0\% & $5.95 \times 10^{-8}$ & $8.44 \times 10^{-7}$ \\
			
			\midrule
			50\% & BF16 & 0.54 & $\mathbf{7.88 \times 10^{5}}$ & - & $7.88 \times 10^{5}$ & $1.10 \times 10^{-2}$ \\
			50\% & FP16 & 0.50 & $\mathbf{1.10 \times 10^{5}}$ & - & $1.10 \times 10^{5}$ & $1.12 \times 10^{-3}$ \\
			50\% & FP32 & 1.81 & $\mathbf{1.68 \times 10^{2}}$ & - & $1.68 \times 10^{2}$ & $4.45 \times 10^{-7}$ \\
			50\% & FP64  & 2.70 & - & - & - & - \\
			50\% & BF16$\rightarrow$BF16 & 33.10 & $\mathbf{4.92 \times 10^{2}}$ & 100\% & $4.92 \times 10^{2}$ & $6.49 \times 10^{-2}$ \\
			50\% & FP16$\rightarrow$FP16 & 0.68 & $\mathbf{1.58 \times 10^{1}}$ & 0.01\% & $1.58 \times 10^{1}$ & $4.57 \times 10^{-3}$ \\
			50\% & FP32$\rightarrow$FP32 & 1.47 & $7.57 \times 10^{-7}$ & 0.01\% & $6.78 \times 10^{-7}$ & $7.57 \times 10^{-7}$ \\
			50\% & BF16$\rightarrow$FP32 & 33.05 & $1.00 \times 10^{-6}$ & 100\% & $6.78 \times 10^{-7}$ & $1.00 \times 10^{-6}$ \\
			50\% & BF16$\rightarrow$FP64 & 32.91 & $5.96 \times 10^{-8}$ & 100\% & $5.83 \times 10^{-8}$ & $5.96 \times 10^{-8}$ \\
			50\% & FP16$\rightarrow$FP32 & 0.67 & $8.50 \times 10^{-4}$ & 0.01\% & $6.78 \times 10^{-7}$ & $8.50 \times 10^{-4}$ \\
			50\% & FP16$\rightarrow$FP64 & 0.67 & $8.50 \times 10^{-4}$ & 0.01\% & $5.83 \times 10^{-8}$ & $8.50 \times 10^{-4}$ \\
			50\% & FP32$\rightarrow$FP64 & 1.49 & $7.57 \times 10^{-7}$ & 0.01\% & $5.83 \times 10^{-8}$ & $7.57  \times 10^{-7}$ \\
			
			\midrule
			100\% & BF16 & 0.54 & $\mathbf{9.19 \times 10^{5}}$ & - & $9.19 \times 10^{5}$ & $9.29 \times 10^{-3}$ \\
			100\% & FP16 & 0.51 & $\mathbf{1.27 \times 10^{5}}$ & - & $1.27 \times 10^{5}$ & $1.13 \times 10^{-3}$ \\
			100\% & FP32 & 1.81 & $\mathbf{1.37 \times 10^{2}}$ & - & $1.37 \times 10^{2}$ & $4.73 \times 10^{-7}$ \\
			100\% & FP64 & 2.71 & - & - & - & - \\
			100\% & BF16$\rightarrow$BF16 & 33.15 & $\mathbf{4.31 \times 10^{2}}$ & 100\% & $4.31 \times 10^{2}$ & $6.49 \times 10^{-2}$ \\
			100\% & FP16$\rightarrow$FP16 & 0.68 & $\mathbf{1.43 \times 10^{1}}$ & 0.02\% & $1.43 \times 10^{1}$ & $3.89 \times 10^{-3}$ \\
			100\% & FP32$\rightarrow$FP32 & 1.47 & $8.22 \times 10^{-7}$ & 0.02\% & $8.05 \times 10^{-7}$ & $8.22 \times 10^{-7}$ \\
			100\% & BF16$\rightarrow$FP32 & 33.13 & $9.72 \times 10^{-7}$ & 100\% & $8.05 \times 10^{-7}$ & $9.72 \times 10^{-7}$ \\
			100\% & BF16$\rightarrow$FP64 & 32.98 & $5.96 \times 10^{-8}$ & 100\% & $5.74 \times 10^{-8}$ & $5.96 \times 10^{-8}$ \\
			100\% & FP16$\rightarrow$FP32 & 0.68 & $8.39 \times 10^{-4}$ & 0.02\% & $8.05 \times 10^{-7}$ & $8.39 \times 10^{-4}$ \\
			100\% & FP16$\rightarrow$FP64 & 0.67 & $8.39 \times 10^{-4}$ & 0.02\% & $5.74 \times 10^{-8}$ & $8.39 \times 10^{-4}$ \\
			100\% & FP32$\rightarrow$FP64 & 1.49 & $8.22 \times 10^{-7}$ & 0.02\% & $5.74 \times 10^{-8}$ & $8.22  \times 10^{-7}$ \\
			\bottomrule
		\end{tabular}
	}
\end{table}

We next study the sensitivity of the fallback mechanism in mixed near-far cases. 
\begin{revision}
{The matrices again contain $5,000$ vectors of dimension $128$. For a fraction $f$ of the points, we set} $C_i=P_i+10^{-4}\mathcal N(0,I)$; the remaining points are independent Gaussian vectors. If each near-duplicate contributes one matched near pair, the near-pair fraction is $f/5,000$. In particular, $100\%$ near-duplicate points correspond to only $0.02\%$ of all distance entries. This experiment therefore tests detection of sparse near pairs, not a workload in which all pairwise distances are small.
The implementation separates near and far entries using the FP64 reference \emph{squared} distance and the threshold $\max(10r\varepsilon^2,10^{-10})$, equal to $1.28\times10^{-5}$ for the displayed case. Timings use milliseconds. The noise sweep covers $\varepsilon\in\{10^{-2},10^{-3},\ldots,10^{-6}\}$; Table~\ref{tab:dist-streess-test} retains the $\varepsilon=10^{-4}$ case. A displayed $0\%$ can hide a small nonzero fallback rate; for example, the matched pairs at $f=10\%$ constitute $0.002\%$ of the matrix.
\end{revision}

We first observe that uniform low-precision methods are highly sensitive to the presence of near-point pairs.  When only $10\%$ of the data are near points, the maximum relative errors of uniform FP16  and BF16 explode to much greater than one, respectively.  Uniform FP32 also becomes unreliable for these near pairs, with errors above one.  The  far-point errors remain small, indicating that the instability is caused mainly by cancellation in the nearly coincident distance computations. Such cancellation can affect individual point-to-center distances; its relationship to clustering convergence is discussed below.

The mixed-precision methods with fallback substantially improve robustness. For FP16$\rightarrow$FP32 and FP16$\rightarrow$FP64, the fallback rate remains very small, at most $0.02\%$, while the near-point errors stay at the level of about $10^{-8}$ to $10^{-6}$.  The overall maximum error is then dominated by the far-point FP16 computation and remains around $10^{-3}$.  In contrast, BF16-based fallback is triggered for essentially all entries in this experiment. This gives stable accuracy when the correction precision is FP32 or FP64, but at a much higher runtime cost.  The execution times follow a trend similar to that in the previous experiments. The mixed-precision methods FP16$\rightarrow$FP32 and FP16$\rightarrow$FP64 consistently achieve approximately $2.7\times$ speedup over the FP32 \texttt{cdist} routine of \texttt{PyTorch}.

\begin{revision}
{These results show improved accuracy for the sparse near pairs while most distances retain the initial expanded evaluation. The experiment does not measure performance when a large fraction of all pairwise distances requires correction.} For clustering, numerical sensitivity depends both on cancellation in individual distances and on small gaps between competing centers. Near convergence, point-to-center distances need not approach zero because within-cluster variation can remain nonzero. Therefore, empirical convergence behavior should be assessed separately from distance-kernel accuracy.
\end{revision}

\subsection{K-means clustering}\label{sec:clustering-implementation}
Our k-means implementation here follows a hybrid CPU-GPU architecture where the lightweight main loop and convergence checking run on CPU, while CUDA kernels handle distance computation and center updates. For comparison, we also include standard uniform-precision kernels using FP16, BF16, FP32, and FP64 (without fallback).
\begin{revision}
Here, distances are computed using the routines in Section~\ref{sec:gpu-arithmetic}. The mixed-precision clustering methods use the Original test described in Section~\ref{sec:implemented-threshold}; the FP32 and FP64 comparisons use PyTorch \texttt{cdist} followed by squaring.

The FP16$\rightarrow$FP64 and BF16$\rightarrow$FP64 configurations convert assignment inputs and updated centers to FP16 and BF16, respectively, before widening them to FP32 for distance evaluation. FP64 recomputation cannot recover information lost in this conversion. The FP16$\rightarrow$FP32, BF16$\rightarrow$FP32 and FP32$\rightarrow$FP64 configurations store centers in FP32. Center sums and division use FP32 except in the FP64 baseline. Each center coordinate is summed sequentially, and empty clusters are reinitialized with randomly selected data points.

Unlike Algorithm~\ref{alg:mp-kmeans}, these configurations use different center-storage precisions. Rounding the stored centers can affect subsequent assignments and the stopping test of center displacement. The reported iteration counts therefore depend on both center storage and fallback arithmetic, so fewer iterations do not by themselves demonstrate a benefit from higher-precision fallback. Iteration stops when the Frobenius norm of the center displacement, evaluated in FP32, is below $\tau=10^{-8}$, or after $300$ iterations. Final assignments use the same distance routine as the iteration.

\end{revision}
Clustering quality is evaluated using the \sse{} of~\eqref{eq:sse}, Adjusted Rand Index (ARI) and Adjusted Mutual Information (AMI)~\cite{JMLR:v11:vinh10a}. ARI and AMI compare the computed clustering with the ground-truth labels of the synthetic Gaussian data. The reported \sse{} values are computed using each method's distance kernel, without reevaluation in a common precision.

\subsubsection{Clustering on Gaussian blobs}

We evaluate the performance of distinct distance kernels with different precision configurations in k-means clustering on synthetic Gaussian blob datasets. Tables~S4.1--S4.4 from Supplementary Section S4 present simulated k-means performance with $n=100{,}000$ samples, feature dimensions $r\in\{10,128\}$, and $k\in\{100,500\}$ clusters. The reported speedups are relative to the FP64 GPU configuration, whose distance calculation uses PyTorch; values below one indicate a slowdown. Ratios are computed from the displayed runtimes and reported to two decimal places. For Gaussian clustering, the first of four runs is excluded from timing averages but included in the iteration-count and quality averages. Table~\ref{tab:km-overview} summarizes the $\rho=5$ cases; Supplementary Section~S4 gives the full results.

\begin{table}[t]
\centering
\caption{Gaussian clustering runtime overview at $\rho=5$ for the mixed variants. Each entry gives seconds (speedup relative to the FP64 GPU configuration in that column). Ratios use the displayed times; all other $\rho$ settings and quality metrics are reported in Supplementary Section~S4.}
\label{tab:km-overview}
\small\renewcommand{\arraystretch}{1.12}\resizebox{0.73\linewidth}{!}{
\begin{tabular}{lrrrr}
\toprule
& \multicolumn{2}{c}{$r=10$}&\multicolumn{2}{c}{$r=128$}\\
\cmidrule(lr){2-3}\cmidrule(lr){4-5}
Method & $k=100$ & $k=500$ & $k=100$ & $k=500$\\
\midrule
BF16 & 1.98 (0.24$\times$) & 4.03 (0.09$\times$) & 1.76 (0.22$\times$) & 2.40 (0.15$\times$)\\
FP16 & 1.97 (0.24$\times$) & 4.01 (0.09$\times$) & 2.41 (0.16$\times$) & 2.42 (0.14$\times$)\\
FP32 & 0.45 (1.04$\times$) & 0.34 (1.12$\times$) & 0.30 (1.30$\times$) & 0.29 (1.21$\times$)\\
FP64 & 0.47 (1.00$\times$) & 0.38 (1.00$\times$) & 0.39 (1.00$\times$) & 0.35 (1.00$\times$)\\
BF16$\rightarrow$FP32 & 0.47 (1.00$\times$) & 0.37 (1.03$\times$) & 0.94 (0.41$\times$) & 1.47 (0.24$\times$)\\
FP16$\rightarrow$FP32 & 0.44 (1.07$\times$) & 0.34 (1.12$\times$) & 0.36 (1.08$\times$) & 0.27 (1.30$\times$)\\
BF16$\rightarrow$FP64 & 0.25 (1.88$\times$) & 0.43 (0.88$\times$) & 1.09 (0.36$\times$) & 1.34 (0.26$\times$)\\
FP16$\rightarrow$FP64 & 0.36 (1.31$\times$) & 0.34 (1.12$\times$) & 0.31 (1.26$\times$) & 0.26 (1.35$\times$)\\
FP32$\rightarrow$FP64 & 0.44 (1.07$\times$) & 0.33 (1.15$\times$) & 0.29 (1.34$\times$) & 0.24 (1.46$\times$)\\
\bottomrule
\end{tabular}}
\end{table}

The no-fallback FP16/BF16 configurations have long iteration sequences in these experiments. Although the individual distance computations are cheaper in low precision, the resulting loss of accuracy increases the number of k-means iterations substantially.
In several cases, uniform FP16 reaches the maximum number of $300$ iterations, and BF16 also requires many more iterations than FP32 or FP64. Consequently, the total runtime of the uniform low-precision methods is often much larger than that of FP32 or FP64.  In contrast, the mixed-precision variants generally recover iteration counts and clustering quality close to those of uniform FP32 and FP64.  
The AMI and ARI values are typically very close to the FP64 reference values, and the final \sse\ values are also comparable. 
\begin{revision}
Table~\ref{tab:km-overview} shows that gains are modest and problem dependent. Using displayed times, FP16$\rightarrow$FP64 with $\rho=5$ has speedups of approximately $1.31$, $1.12$, $1.26$, and $1.35$ in the four configurations. The FP32 baseline is competitive without threshold tuning. BF16$\rightarrow$FP32 at $r=128,k=500$, for example, takes $1.47$s compared with $0.35$s for FP64, so its speedup is only $0.24\times$. The configurations differ in center storage as specified in Section~\ref{sec:clustering-implementation}.

The choice $\rho=5$ is an empirical compromise among the tested values $\{5,10,50\}$; Supplementary Section~S4 reports the full threshold sweep and shows that larger values can substantially increase runtime, especially for $r=128$.

The results show that many mixed-precision configurations have AMI and ARI values close to the FP64 baseline. However, an iteration count of 300 does not necessarily mean that the center displacement stopping criterion was satisfied.
\end{revision}

\subsubsection{Image segmentation}
In this section we test the k-means methods for image segmentation.
Given an input image of height $H$ and width $W$, we treat each pixel as one sample and construct a feature matrix $P \in \mathbb{R}^{N \times 5}$ with $N = H W$.
Each row of $P$ concatenates color and spatial information of one pixel, namely, $(R,G,B,x,y)$, where $(R,G,B)$ are the normalized color channels and $(x,y)$ are normalized pixel coordinates. 
Unless stated otherwise, each feature dimension is rescaled using min-max normalization so that the color and spatial coordinates have comparable magnitudes in the distance computations.

\begin{table*}
	\centering
	\caption{Image-segmentation k-means results on two ImageNet images with
		$k\in\{8,32\}$ clusters.}
	\label{tab:img-seg}
	\setlength{\tabcolsep}{4.2pt}
	\renewcommand{\arraystretch}{1.12}
	\resizebox{\textwidth}{!}{
		\begin{tabular}{l rrr  rrr rrr rrr}
			\toprule
			 & \multicolumn{6}{c}{\texttt{ILSVRC2012\_val\_00011129}}
			& \multicolumn{6}{c}{\texttt{ILSVRC2012\_val\_00006582}} \\
			\cmidrule(lr){2-7} \cmidrule(lr){8-13}
			 & \multicolumn{3}{c}{$k=8$}
			 & \multicolumn{3}{c}{$k=32$}
			& \multicolumn{3}{c}{$k=8$}
			& \multicolumn{3}{c}{$k=32$} \\
			\cmidrule(lr){2-4} \cmidrule(lr){5-7}
			\cmidrule(lr){8-10} \cmidrule(lr){11-13}
			Method & \texttt{Diff} & Iter. & Time $(s)$
			& \texttt{Diff} & Iter. & Time $(s)$
			& \texttt{Diff} & Iter. & Time $(s)$
			& \texttt{Diff} & Iter. & Time $(s)$ \\
			\midrule
			FP16
			& 0.12\% & 300.0 & 3.71
			& 7.62\% & 300.0 & 3.67
			& 0.29\% & 300.0 & 3.74
			& 1.31\% & 300.0 & 3.65 \\
			
			BF16
			& 1.14\% & 300.0 & 3.72
			& 18.47\% & 300.0 & 3.68
			& 1.48\% & 300.0 & 3.77
			& 32.52\% & 300.0 & 3.67 \\
			
			FP32
			& 0 & 34.7 & 0.46
			& 0.02\% & 107.3 & 1.32
			& 0.04\% & 92.0 & 1.31
			& 0.01\% & 145.3 & 1.78 \\
			
			FP64
			& - & 33.7 & 0.42
			& - & 107.0 & 1.30
			& - & 91.3 & 1.16
			& - & 145.7 & 1.78 \\
			
			FP16$\rightarrow$FP32
			& 0.08\% & 300.0 & 3.71
			& 0.15\% & 300.0 & 3.68
			& 0.14\% & 241.0 & 2.99
			& 0.14\% & 300.0 & 3.68 \\
			
			BF16$\rightarrow$FP32
			& 0.01\% & 34.7 & 0.43
			& 0.02\% & 107.3 & 1.33
			& 0.04\% & 92.3 & 1.17
			& 0.03\% & 146.0 & 1.81 \\
			
			FP16$\rightarrow$FP64
			& 0.08\% & 300.0 & 3.81
			& 0.18\% & 300.0 & 3.70
			& 0.08\% & 218.3 & 2.88
			& 0.24\% & 234.3 & 2.88 \\
			
			BF16$\rightarrow$FP64
			& 0.36\% & 21.7 & 0.30
			& 14.26\% & 46.0 & 0.58
			& 1.01\% & 49.0 & 0.70
			& 2.31\% & 80.7 & 1.01 \\
			
			FP32$\rightarrow$FP64
			& 0 & 34.7 & 0.43
			& 0.02\% & 107.7 & 1.31
			& 0.04\% & 92.0 & 1.15
			& 0.03\% & 145.7 & 1.77 \\
			\bottomrule
		\end{tabular}
	}
\end{table*}

We run k-means with cluster number $k \in \{8, 32\}$ to capture both coarse and fine-grained segmentations. The k-means method iteratively alternates between assigning each pixel feature to its nearest centroid in the normalized feature space and updating centroids by averaging the features of assigned samples. After termination, we reshape the resulting assignment vector into a $H \times W$ label map. Image-clustering runtimes are averaged over measured runs after a separate untimed warm-up, with GPU synchronization before and after fitting.  Table~\ref{tab:img-seg} summarizes the results for two images from ImageNet \cite{ddsl09}.  The method in uniform FP64 serves as the reference for fidelity comparisons, and \texttt{Diff} measures the permutation-invariant percentage of pixel labels that differ from this reference after optimal one-to-one alignment of cluster IDs (to account for the label-permutation symmetry of k-means).

The uniform low-precision kernels are unreliable for this task.  Both BF16 and FP16 reach the maximum number of iterations in all four test cases. The degradation is especially pronounced for BF16 at $k=32$, where the disagreement rates reach $18.47\%$ and $32.52\%$ for the two images.  This indicates that using reduced precision alone can substantially change the segmentation when the number of clusters is increased.
On the other hand, uniform FP32 closely aligns with the FP64 reference in all metrics.  
The mixed-precision variants show more varied behavior. This indicates that, under the present feature scaling, FP32 is already sufficient to reproduce the FP64 segmentation to high accuracy.

The mixed-precision methods exhibit different behavior depending on the base precision.  The FP16-based variants improve the agreement with the FP64 reference relative to uniform FP16, but they still often require many iterations.  In particular, FP16$\rightarrow$FP32 reaches the maximum iteration count in three of the four cases, and FP16$\rightarrow$FP64 reaches it in two cases.  
Hence, although FP64 fallback reduces the disagreement rate, it does not fully remove the convergence difficulty caused by the FP16 base computation. The BF16-based variants show a more nuanced trade-off.  BF16$\rightarrow$FP32 closely matches the FP64 reference with comparable iteration counts and runtimes to FP32 and FP64.
BF16$\rightarrow$FP64 often terminates in fewer iterations and is faster in this table, and stores updated centers in BF16. Its timing advantage therefore cannot be attributed solely to FP64 correction, and it can accompany noticeably different segmentations. The most pronounced case is \texttt{ILSVRC2012\_val\_00011129} with $k=32$: BF16$\rightarrow$FP64 takes $0.58$s versus $1.30$s for FP64, but its disagreement rate is $14.26\%$. The faster run therefore produces a different segmentation. Since no ground-truth segmentation is used, \texttt{Diff} measures agreement with the reference rather than semantic segmentation quality.
The FP32$\rightarrow$FP64 variant is the most conservative mixed-precision method.  It essentially reproduces the FP64 reference, with disagreement rates no larger than $0.04\%$, while achieving runtimes comparable to or slightly smaller than those of FP32 and FP64.  Since uniform FP32 is already highly accurate in this experiment, the benefit of adding FP64 fallback to FP32 is modest.

Overall, our experiments show that image segmentation by k-means is sensitive to uniform low-precision (FP16 and BF16) arithmetic, especially as the number of clusters $k$ increases, and k-means using the mixed-precision strategy substantially reduces the \texttt{Diff} for image segmentation compared to the low-precision arithmetic. Therefore, FP16 and BF16 should not be used as uniform precisions for this task, since their lower arithmetic cost is outweighed by poor convergence and larger segmentation discrepancies. The lower-precision mixed-precision variants can improve runtime in some cases, but their accuracy depends on the image and the number of clusters.

\section{Conclusions}\label{sec:conclusion}
In this paper we studied the use of mixed-precision arithmetic in Lloyd's algorithm for $k$-means clustering, with a particular focus on the two main numerical components of the iteration{--}the Euclidean distance computation used in the assignment step and the center update. The analysis gives theoretical insights into the use of the two distance computation routines and sufficient conditions under which rounded center updates still decrease the  Lloyd’s objective when the assignments are fixed. Regarding the distance computation, the extended model analyzed in Section~\ref{sec:accum-storage} isolates the accumulation error from those relating to intermediate storage and input conversion. This distinction is essential for interpreting FP16/BF16 GPU implementations that accumulate in FP32 and subsequently round intermediate results back to a low-precision format.

In addition, we provided CUDA implementations of the resulting mixed-precision $k$-means variants and compared them with uniform-precision implementations and standard library routines. The empirical results on distance evaluation show a favorable trade-off between speed and accuracy; the use of low-precision distances leads to significant speedups, while selective recomputation in higher precision deals with the operations prone to cancellation.  The ideas and methods developed in this work are expected to benefit a broad class of Euclidean-distance-based algorithms, including, but not limited to, those used in information retrieval and retrieval-augmented generation (RAG)~\cite{chgu24b, jds10,lewis2020retrieval}, data clustering~\cite{chgu24a, eksx96}, manifold learning~\cite{tdl00}, and time-series analysis~\cite{11397443}.

Our distance computation method can be substantially slower than a uniform FP64 implementation of the expanded formula when the cost of the fallback scheme outweighs the savings from low-precision computation. Future work includes investigating how the underlying data distribution affects the recomputation rate and hence the overall performance.

\bibliographystyle{siamplain}
\bibliography{references}
\end{document}


\maketitle
\begin{revision}
This supplementary material accompanies the paper \emph{Mixed-precision Lloyd's algorithm for k-means clustering} and provides supporting analysis and numerical results not included in the main text.

Section~\ref{app:reference-test} details a directed-rounding implementation of the reliability test that accounts separately for input conversion, FP32 accumulation, intermediate storage, and gradual underflow; it also gives the proof of the cluster-center summation error bound stated in the main paper. 
Section~\ref{supp:details} provides further details of the GPU timing protocol and arithmetic paths. 
Section~\ref{app:uploaded-results} reports additional accuracy checks, assignment comparisons, and implementation evidence for the storage/threshold experiments in main Sections~6.1.4--6.1.5.
Finally, Section~\ref{app:full-gaussian} gives the complete Gaussian clustering results for $\rho\in\{5,10,50\}$, extending the $\rho=5$ experiment presentation in main Section~6.3. 

Unless otherwise stated, the notation, data, hardware, and experimental protocols follow the main paper.
The software and experimental code are available in the accompanying repository at \url{https://github.com/inEXASCALE/mp-kmeans}.
\end{revision}

\section{Additional rounding-error analysis}\label{app:reference-test}

\begin{revision}
\subsection{Reliability test with separate accumulation and storage precisions}\label{supp:reference-accumulation}
This section specializes the error bound of main Section~4.4 to sequential FP32 accumulation, adds allowances for gradual underflow, and describes its evaluation with directed rounding. The cuBLAS experiments in Supplementary Section~\ref{app:uploaded-results} use the same bound subject to the library inner-product error assumption in main equation~(4.19), augmented by the underflow allowance described below; they do not measure the sequential cross-product implementation.


Let $x,y$ be finite original FP32 vectors and let $x_q=\widetilde x$, $y_q=\widetilde y$ their low-format conversions widened exactly to FP32. Compute upper bounds
\begin{equation*}
N_x\ge\norm{x},\quad Q_x\ge\norm{x_q},\quad E_x\ge\norm{x-x_q},
\end{equation*}
and similarly for $y$. These use FP64 multiplication, addition and square root rounded upward; each conversion residual is enclosed by the downward- and upward-rounded subtraction. With $E=E_x+E_y$, the triangle inequality gives
\begin{equation}\label{eq:reference-conversion}
\bigl|\norm{x_q-y_q}^2-\norm{x-y}^2\bigr|
\le C_*:=2(N_x+N_y)E+E^2.
\end{equation}
Thus the actual conversion residual is retained even when relative-error assumptions fail for subnormal inputs.

The reference norm and cross terms use sequential FP32 fused multiply-add, in coordinate order, without flush-to-zero. Let $u_a=2^{-24}$ and $\eta_a=2^{-149}$ denote the FP32 unit roundoff and full subnormal spacing. Assuming $ru_a<1$, set
\begin{equation*}
g=\frac{ru_a}{1-ru_a},\qquad a_* = \frac{r\eta_a}{1-ru_a}.
\end{equation*}
For each of $(a,b)=(x_q,x_q),(x_q,y_q),(y_q,y_q)$ and finite intermediate values, the inner-product error is bounded by $g\,|a|^T|b|+a_*$. This follows by propagating the per-operation model $|\mathrm{fl}(z)-z|\le u_a|z|+\eta_a$ through the recurrence; the full subnormal spacing is a conservative absolute allowance.

Let $u_s$ and $\eta_s$ be the unit roundoff and full subnormal spacing of the stored intermediates. For FP16, BF16 and FP32, respectively, these pairs are $(2^{-11},2^{-24})$, $(2^{-8},2^{-133})$ and $(2^{-24},2^{-149})$. Following the storage and final-rounding analysis in main Section~4.4, set
\begin{align*}
\alpha_*&=(1+g)(1+u_s)-1,& w_*&=(1+u_s)a_*+\eta_s,\\
g_2&=\frac{2u_a}{1-2u_a},& \beta_*&=(1+\alpha_*)(1+g_2)-1.
\end{align*}
The stored inner products have error coefficient $\alpha_*$ and additive allowance $w_*$. The final expression is formed with two explicit FP32 additions/subtractions; multiplication of the cross term by two is exact when finite. Since
\begin{equation*}
\norm{x_q}^2+2|x_q|^T|y_q|+\norm{y_q}^2
\le 2(Q_x^2+Q_y^2),
\end{equation*}
a total error bound, also valid after clamping to zero, is
\begin{equation}\label{eq:reference-total}
B_* = C_*+2\beta_*(Q_x^2+Q_y^2)
       +4(1+g_2)w_*+\frac{2\eta_a}{1-2u_a}.
\end{equation}
For FP32 storage the accumulated FP32 value is unchanged; the chosen $u_s,\eta_s$ remain conservative. 

To preserve the upper bound during its FP64 evaluation, every operation constructing an upper bound rounds upward, except positive denominators, which round downward. The implementation restricts $r<2^{23}$ and sends nonfinite expanded results or bounds to direct recomputation.

The clamped value $\widehat D$ is accepted if
\begin{equation}\label{eq:reference-accept}
\widehat D>\mathrm{round}_{+}(\rho B_*),\qquad \rho>1,
\end{equation}
Here $\mathrm{round}_{+}$ denotes rounding toward $+\infty$ in FP64, ensuring that the computed threshold is at least $\rho B_*$. As in main Section~4.4, the acceptance test implies $D=\norm{x-y}^2\ge\widehat D-B_*> (\rho-1)B_*$ and $|\widehat D-D|/D<1/(\rho-1)$. At $\rho=5$ this is a coarse $25\%$ sufficient tolerance for an accepted distance. This guarantee differs from the bound $2/(\widetilde\rho-2)$ in main equation~(4.17) and does not automatically apply to the empirical thresholds in main Section~6.1.2. It does not guarantee the correct nearest center in k-means clustering.

{Distances accepted by the test are already FP32 and can be stored in FP32 or widened to FP64 exactly. Rejected distances are recomputed from the original FP32 data in FP32 or FP64. Returning FP64 recomputation in FP32 introduces output rounding and can underflow; the preceding relative-error bound applies to distances accepted by the test.}

{For cuBLAS, replacing $g$ by $\max(\varepsilon_a,g)$ assumes that the cross-product errors satisfy the corresponding relative and additive bounds. The default choice $\varepsilon_a=g$ is an assumption for cuBLAS, whereas the sequential bound is derived above. Supplementary Section~\ref{app:uploaded-results} reports results under this assumption. Its timings include computation of norm and distance-error bounds, FP32 copies of quantized inputs, identification of rejected entries and direct recomputation. Kernel names and observed errors provide implementation evidence, but do not prove the assumed cuBLAS bound.}
\end{revision}

\begin{revision}
\subsection{Proof of the cluster-center summation bound}\label{supp:center-proof}
This subsection gives the proof of Lemma 4.4 in the main paper. Let $S_i=\{p_1,\ldots,p_{m_i}\}$ be a nonempty cluster, and suppose that each coordinate of $\sum_{j=1}^{m_i}p_j$ is formed by recursive summation in precision $u$ and then divided by $m_i$ in the same precision, without overflow or underflow. Assume $m_i u<1$ and set $\gamma_j=ju/(1-ju)$. With
\begin{equation*}
\mu_i=\frac{1}{m_i}\sum_{j=1}^{m_i}p_j,
\qquad
\bar a_i=\frac{1}{m_i}\sum_{j=1}^{m_i}|p_j|,
\end{equation*}
we prove $|\widehat\mu_i-\mu_i|\le\gamma_{m_i}\bar a_i$.
\begin{proof}
Label the coordinates of data point $p_j$ as $p_j=(p_{j,1},\ldots,p_{j,r})^T$, $j=1,\ldots,m_i$. For a fixed coordinate \(\ell=1,\ldots,r\), define
\begin{equation*}
    s_\ell:=\sum_{j=1}^{m_i}p_{j,\ell},
    \quad
    \mu_{i,\ell}=\frac{s_\ell}{m_i}.
\end{equation*}
The computed coordinate sum is obtained from the recursive summation
\begin{equation*}
    \widehat s_{1,\ell}:=p_{1,\ell},
    \quad
    \widehat s_{k,\ell}:=
    \fl(\widehat s_{k-1,\ell}+p_{k,\ell}),
    \quad k=2,\ldots,m_i.
\end{equation*}
Thus \(\widehat s_\ell=\widehat s_{m_i,\ell}\). From the standard
floating-point arithmetic model~(4.1), there exist
\(\varepsilon_{k,\ell}\), \(k=2,\ldots,m_i\), with $|\varepsilon_{k,\ell}|\le u$, such that
\begin{equation*}
    \widehat s_{k,\ell} =
    (\widehat s_{k-1,\ell}+p_{k,\ell})(1+\varepsilon_{k,\ell}),
    \quad k=2,\ldots,m_i.
\end{equation*}
Expanding the recursion gives
\begin{equation*}
\begin{aligned}
    \widehat s_{2,\ell}
    &=
    p_{1,\ell}(1+\varepsilon_{2,\ell})
    +p_{2,\ell}(1+\varepsilon_{2,\ell}), \\
    \widehat s_{3,\ell}
    &=
    p_{1,\ell}(1+\varepsilon_{2,\ell})(1+\varepsilon_{3,\ell})
    +p_{2,\ell}(1+\varepsilon_{2,\ell})(1+\varepsilon_{3,\ell}) +p_{3,\ell}(1+\varepsilon_{3,\ell}),
\end{aligned}
\end{equation*}
and, in general,
\begin{equation*}
    \widehat s_\ell
    =
    p_{1,\ell}\prod_{k=2}^{m_i}(1+\varepsilon_{k,\ell})
    +
    \sum_{j=2}^{m_i}
    p_{j,\ell}\prod_{k=j}^{m_i}(1+\varepsilon_{k,\ell}).
\end{equation*}
Hence we can write
\begin{equation*}
    \widehat s_\ell
    =
    \sum_{j=1}^{m_i}p_{j,\ell}(1+\xi_{j,\ell}), \quad 
    |\xi_{j,\ell}| \le
    \gamma_{m_i-1}, \quad j=1,\ldots,m_i .
\end{equation*}
The computed centroid coordinate is obtained by one floating-point division,
\begin{equation*}
    \widehat\mu_{i,\ell}
    =
    \fl\left(\frac{\widehat s_\ell}{m_i}\right)
    =
    \frac{\widehat s_\ell}{m_i}(1+\eta_\ell),
    \quad
    |\eta_\ell|\le u.
\end{equation*}
Substituting the expression for \(\widehat s_\ell\), we obtain
\begin{equation*}
    \widehat\mu_{i,\ell}
    =
    \frac{1}{m_i}
    \sum_{j=1}^{m_i}
    p_{j,\ell}(1+\xi_{j,\ell})(1+\eta_\ell).
\end{equation*}
Since $(1+\xi_{j,\ell})(1+\eta_\ell) = 1+\theta_{j,\ell}$, and since the product now contains at most \(m_i\) rounding factors, we have $|\theta_{j,\ell}|\le \gamma_{m_i}$, $j=1,\ldots,m_i$. Therefore
\begin{equation*}
    \widehat\mu_{i,\ell} =
    \frac{1}{m_i}
    \sum_{j=1}^{m_i}
    p_{j,\ell}(1+\theta_{j,\ell})
    =
    \mu_{i,\ell}+\Delta\mu_{i,\ell}, \quad 
    \Delta\mu_{i,\ell} =
    \frac{1}{m_i}
    \sum_{j=1}^{m_i}p_{j,\ell}\theta_{j,\ell}.
\end{equation*}
Taking absolute values for every coordinate gives the componentwise bound. Consequently, the first norm bound follows directly, and the second follows by applying the Cauchy--Schwarz inequality.
\end{proof}

\end{revision}

\begin{revision} 
\section{Further experiment details}\label{supp:details}

\subsection{Additional arithmetic and build details}\label{supp:arithmetic}
The $L\rightarrow L$ ($L$ is FP16 or BF16) methods initially evaluate the expanded formula and then recompute rejected distances on the quantized $L$ inputs using \texttt{\_\_hsub}, \texttt{\_\_hmul}, and \texttt{\_\_hadd} with an $L$ accumulator, before returning FP32. {Consequently, these methods can have different error behavior} from both the no-fallback $L$ implementation and the $L\rightarrow$FP32 variant. The main BF16 GEMM is a library operation; the surrounding norm and formation kernels explicitly use FP32 arithmetic. For the profiled \texttt{sm\_80} build, the executed fallback kernel is identified in an Nsight Systems trace and its matching extension disassembly contains \texttt{HFMA2.BF16\_\allowbreak V2} subtraction, square and accumulation idioms. {This establishes native BF16 arithmetic in the compiled recomputation kernel;} Supplementary Section~\ref{app:uploaded-results} gives the evidence and its scope.

The FP32 arrow variants use FP32 norm accumulation, \texttt{torch::mm}, and FP32 stored intermediates; the right-hand format selects direct FP32 or FP64 fallback. PyTorch's matrix-multiplication policy and library reductions are separate from the same-precision sequential model. {The distance extension is compiled with} \texttt{-O3} and \texttt{--use\_fast\_math}, so the standard no-underflow model is not an unconditional description of all generated instructions. The repository also includes tiled and WMMA implementations, and provides Python scripts that call the untiled interfaces described above. As such, the reported timings therefore describe the untiled interfaces.

\subsection{GPU timing on distance computation}\label{supp:protocol}

For the distance benchmarks in main Sections~6.1.3 and~6.2, the timing routine executes one untimed warm-up call and then times three calls using CUDA events, with synchronization before the measured block and after the end event. As mentioned in the paper, the reported time is the event interval divided by three, in milliseconds.  Input generation and host-side error evaluation are outside the timed region. Mixed-kernel calls include their internal conversions, allocations, norm computation, GEMM, list construction, fallback launches, and synchronization. The no-fallback low-precision benchmark receives inputs converted before timing.
\end{revision}

\begin{revision}
\section{Storage, reliability tests and execution experiments}\label{app:uploaded-results}
The experiments use the NVIDIA A100 GPU allocation described in the main paper. Main Sections~6.1.4--6.1.5 report the effects of storage precision and threshold choice and the measured execution costs.
{Original uses the empirical threshold coefficients reported in main Section~6.1.2; Extended uses the error bound and test in Supplementary Section~\ref{app:reference-test}, with the stated cuBLAS error assumption.}

{There are 60 distance runs: eight storage/threshold configurations and direct FP32 and FP64 comparisons for two datasets and three seeds.} Each configuration uses three warm-ups followed by ten CUDA-event measurements per seed. Table~{6.4} reports medians of seed-specific timing medians, the largest per-run maximum relative error over seeds, and median fallback percentages. 
\par
\end{revision}

\begin{revision}
\subsection{Additional validation and memory measurements}
For BF16, retaining FP32 intermediates alone does not change the Original test's $100\%$ fallback. Switching to the Extended test lowers fallback to $0\%$ (far) or $0.02\%$ (matched). On far data, the BF16 low-storage runtime changes from $59.619$ to $2.528$ ms, while the worst relative error increases from $5.960\times10^{-8}$ to $6.970\times10^{-3}$. With FP32 storage, the corresponding change is $59.342$ to $2.248$ ms and $5.960\times10^{-8}$ to $2.332\times10^{-3}$. Thus the Extended test accepts less accurate default distances instead of recomputing every pair. Comparing storage formats with the Extended test fixed, FP32 storage reduces the worst far-point error by about a factor of three for both input formats and lowers the measured time by about $11\%$.

For the dataset containing 5,000 matched near pairs among 25 million distance entries, exactly 5,000 distances are recomputed using the direct formula in FP64 for each configuration and seed when the Extended test is used; near-pair relative errors stay below $5.961\times10^{-8}$, consistent with FP32 output rounding after FP64 correction. Far pairs determine the overall maximum errors. Across storage choices and seeds, FP16 Original recomputation fractions range from $0.000876\%$ to $0.000984\%$ on far data and from $0.021000\%$ to $0.021120\%$ on matched data. All 60 runs have finite outputs. In the 48 storage/threshold runs, no distance accepted by the test exceeds its absolute-error bound or relative-error tolerance, and all accepted bounds are finite. 

The far-case BF16 Extended variants disagree with the FP64 reference argmin for $1.66\%$--$2.04\%$ of rows with low storage and $0.52\%$--$0.64\%$ with FP32 storage; BF16 Original has zero disagreement. For the dataset containing matched near pairs, all storage/threshold configurations give the same nearest-center assignments as the direct FP64 reference. {Assignments are compared after one distance evaluation, without subsequent center updates.} Peak incremental allocated memory is $433{,}174{,}528$ bytes {for every storage/threshold configuration in this implementation}, with either storage choice. {The allocation includes FP32 copies of quantized inputs, the cross-product matrix, per-point norm bounds, the per-distance error-bound array and rejection mask; lowering intermediate-storage precision does not reduce this measured peak.}

{Direct FP64 takes} $57.132$ and $57.229$ ms on far and matched data, respectively. Direct FP32 takes $57.141$ and $57.209$ ms. Their peak increments are $200{,}000{,}000$ and $100{,}000{,}256$ bytes, respectively. {These direct-form comparisons use a different kernel and timing protocol} from the expanded-form baseline D05 in main Table~{6.3}; speedups in that table are defined relative to D05.

\subsection{Profiled stage costs}
One Nsight Systems trace is collected per Extended configuration on far data with seed~0, after three warm-ups. {BF16$\rightarrow$BF16 is also profiled for comparison.} Main Table~{6.5} sums kernel durations by stage within the captured call. For the Extended configurations, {forming distances, evaluating error bounds and identifying rejected entries cost} $1.712$--$1.718$ ms, accounting for $66.8\%$--$76.5\%$ of summed kernel time. GEMM accounts for only $5.5\%$--$8.7\%$. {The kernel that scans the rejection mask still takes} $0.239$--$0.242$ ms even when no entry is recomputed: the kernel scans the full mask and returns early for accepted entries. This overhead is distinct from actual direct distance arithmetic.
\par
\end{revision}

\begin{revision}
The FP32-storage GEMMs are slightly slower in these traces. The lower overall cost is instead consistent with avoiding the low-format {conversion of the stored cross-product matrix to FP32}: the largest conversion/copy launch in each low-storage trace costs about $0.299$--$0.304$ ms and the source explicitly widens the stored cross product to FP32. With FP32 storage this conversion is unnecessary. The same-format BF16 fallback costs $32.013$ ms, or $96.8\%$ of summed kernel time. {It operates on quantized BF16 inputs with BF16 accumulation; the BF16 Original configurations in Table~{6.4} instead recompute in FP64 from original FP32 inputs.}

\subsection{Library launches and native BF16 instructions}
The traces identify an \texttt{sm80\_\allowbreak xmma\_\allowbreak gemm} kernel with a \texttt{tensor16x8x16} suffix for FP16/low storage, and an \texttt{ampere\_\allowbreak s16816gemm\_\allowbreak fp16} kernel for FP16/FP32 storage. BF16/low and BF16/FP32 storage select \texttt{ampere\_\allowbreak bf16\_\allowbreak s1688gemm\_\allowbreak bf16} and \texttt{ampere\_\allowbreak s1688gemm\_\allowbreak bf16}, respectively. Full names are retained in the profile summaries. These names are consistent with the requested Ampere Tensor Core GEMM paths. Nsight Systems does not, however, establish their exact executed instruction stream or reduction-error bound.

{For BF16$\rightarrow$BF16, the trace identifies} \texttt{direct\_\allowbreak bf16\_\allowbreak on\_\allowbreak list}. The matching extension disassembly contains \texttt{HFMA2.BF16\_\allowbreak V2} instructions implementing subtraction, squaring and accumulation. This confirms native BF16 arithmetic for that compiled fallback; it does not describe the surrounding norm kernels, GEMM or the FP64 correction path. Full kernel names, disassembly and build metadata are retained in the repository.

All profiled calls completed with finite outputs. The profiles describe single-call stage costs rather than repeated performance estimates. Nsight Compute could not access the performance counters, so no measured DRAM traffic, achieved hardware throughput or quantitative compute-/memory-bound classification is claimed.
\end{revision}

\begin{revision}
\section{Full Gaussian clustering results}\label{app:full-gaussian}
\begin{revision}
The following tables report timings, iteration counts, ARI, AMI and \texttt{SSE} for different Gaussian clustering parameters. Speedup ratios relative to FP64 are computed from the displayed times and shown to two decimal places. The main text summarizes the $\rho=5$ runtimes in Table~{6.7}. \texttt{SSE} is computed using each method's distance routine, as described in Section~6.3. Methods with arrow use the Original empirical test.

The threshold sweep shows that increasing $\rho$ can substantially raise the runtime, especially for $r=128$. For example, at $r=128,k=500$, FP16$\rightarrow$FP32 takes $0.27$ s at $\rho=5$ but $1.66$ s at $\rho=10$. Thus $\rho=5$ is an empirical compromise among the tested values, not a universally optimal setting. Its role in the supplied threshold is distinct from the theoretical safety factor in the reference reliability test discussed in the main paper.
\end{revision}

\begin{table}[!ht]
	\centering
	\caption{
		Performance of k-means on Gaussian blob data with
		$n=100{,}000$, feature dimension $r=10$, and $k=100$ clusters.
		The table compares different distance kernels for $\rho\in\{5,10,50\}$. 
	}\label{tab:km-gaussblob1}
	\renewcommand{\arraystretch}{1.12}\resizebox{0.8\linewidth}{!}{%
		\begin{tabular}{l r r r r r r r}
			\toprule
			Method & $\rho$ & Time (s) & Speedup & Iter & AMI & ARI & SSE \\
			\midrule
			BF16 & - & 1.98 & 0.24 & 300.0 & 0.9621 & 0.8067 & 73,578.38 \\
			FP16 & - & 1.97 & 0.24 & 300.0 & 0.9610 & 0.8014 & 73,046.56 \\
			FP32 & - & 0.45 & 1.04 & 66.0 & 0.9608 & 0.8005 & 73,050.01 \\
			FP64 & - & 0.47 & 1.00 & 66.0 & 0.9608 & 0.8005 & 73,050.00 \\
			\midrule
			\multirow{3}{*}{BF16$\rightarrow$FP32}
			 & 5 & 0.47 & 1.00 & 66.0 & 0.9608 & 0.8005 & 73,050.00 \\
			 & 10 & 0.47 & 1.00 & 66.0 & 0.9608 & 0.8005 & 73,050.00 \\
			 & 50 & 0.47 & 1.00 & 66.0 & 0.9608 & 0.8005 & 73,050.00 \\
			\cdashline{1-8} 
			\multirow{3}{*}{FP16$\rightarrow$FP32}
			 & 5 & 0.44 & 1.07 & 64.8 & 0.9610 & 0.8012 & 73,052.14 \\
			 & 10 & 0.44 & 1.07 & 66.0 & 0.9610 & 0.8012 & 73,053.52 \\
			 & 50 & 0.46 & 1.02 & 66.0 & 0.9608 & 0.8005 & 73,050.00 \\
			\midrule
			\multirow{3}{*}{BF16$\rightarrow$FP64}
			 & 5 & 0.25 & 1.88 & 42.5 & 0.9609 & 0.8010 & 73,057.62 \\
			 & 10 & 0.25 & 1.88 & 42.5 & 0.9609 & 0.8010 & 73,057.62 \\
			 & 50 & 0.25 & 1.88 & 42.5 & 0.9609 & 0.8010 & 73,057.62 \\
			\cdashline{1-8} 
			\multirow{3}{*}{FP16$\rightarrow$FP64}
			 & 5 & 0.36 & 1.31 & 55.0 & 0.9610 & 0.8012 & 73,052.51 \\
			 & 10 & 0.37 & 1.27 & 54.0 & 0.9610 & 0.8012 & 73,053.78 \\
			 & 50 & 0.35 & 1.34 & 50.8 & 0.9608 & 0.8005 & 73,050.73 \\
			\cdashline{1-8} 
			\multirow{3}{*}{FP32$\rightarrow$FP64}
			 & 5 & 0.44 & 1.07 & 66.0 & 0.9608 & 0.8005 & 73,050.00 \\
			 & 10 & 0.44 & 1.07 & 66.0 & 0.9608 & 0.8005 & 73,050.00 \\
			 & 50 & 0.44 & 1.07 & 66.0 & 0.9608 & 0.8005 & 73,050.00 \\
			\bottomrule
		\end{tabular}
	}%
\end{table}

\begin{table}[!ht]
	\centering
	\caption{
		Performance of k-means on Gaussian blob data with
		$n=100{,}000$, feature dimension $r=10$, and $k=500$ clusters.
		The table compares different distance kernels for $\rho\in\{5,10,50\}$. 
	}\label{tab:km-gaussblob2}
	\renewcommand{\arraystretch}{1.12}\resizebox{0.8\linewidth}{!}{%
		\begin{tabular}{l r r r r r r r}
			\toprule
			Method & $\rho$ & Time (s) & Speedup & Iter & AMI & ARI & SSE \\
			\midrule
			BF16 & - & 4.03 & 0.09 & 300.0 & 0.9663 & 0.8101 & 60,338.57 \\
			FP16 & - & 4.01 & 0.09 & 300.0 & 0.9636 & 0.7970 & 61,623.71 \\
			FP32 & - & 0.34 & 1.12 & 25.0 & 0.9636 & 0.7969 & 61,617.53 \\
			FP64 & - & 0.38 & 1.00 & 25.0 & 0.9636 & 0.7969 & 61,617.52 \\
			\midrule
			\multirow{3}{*}{BF16$\rightarrow$FP32}
			 & 5 & 0.37 & 1.03 & 25.0 & 0.9636 & 0.7969 & 61,617.52 \\
			 & 10 & 0.39 & 0.97 & 25.0 & 0.9636 & 0.7969 & 61,617.52 \\
			 & 50 & 0.38 & 1.00 & 25.0 & 0.9636 & 0.7969 & 61,617.52 \\
			\cdashline{1-8}
			\multirow{3}{*}{FP16$\rightarrow$FP32}
			 & 5 & 0.34 & 1.12 & 27.5 & 0.9636 & 0.7969 & 61,596.62 \\
			 & 10 & 0.34 & 1.12 & 27.8 & 0.9636 & 0.7969 & 61,590.41 \\
			 & 50 & 0.36 & 1.06 & 25.0 & 0.9636 & 0.7969 & 61,617.52 \\
			\midrule
			\multirow{3}{*}{BF16$\rightarrow$FP64}
			 & 5 & 0.43 & 0.88 & 27.2 & 0.9632 & 0.7945 & 62,043.20 \\
			 & 10 & 0.45 & 0.84 & 27.2 & 0.9632 & 0.7945 & 62,043.20 \\
			 & 50 & 0.45 & 0.84 & 27.2 & 0.9632 & 0.7945 & 62,043.20 \\
			\cdashline{1-8}
			\multirow{3}{*}{FP16$\rightarrow$FP64}
			 & 5 & 0.34 & 1.12 & 28.0 & 0.9636 & 0.7969 & 61,598.09 \\
			 & 10 & 0.34 & 1.12 & 28.0 & 0.9636 & 0.7969 & 61,591.78 \\
			 & 50 & 0.36 & 1.06 & 25.5 & 0.9635 & 0.7961 & 61,711.35 \\
			\cdashline{1-8}
			\multirow{3}{*}{FP32$\rightarrow$FP64}
			 & 5 & 0.33 & 1.15 & 25.0 & 0.9636 & 0.7969 & 61,617.53 \\
			 & 10 & 0.33 & 1.15 & 25.0 & 0.9636 & 0.7969 & 61,617.53 \\
			 & 50 & 0.33 & 1.15 & 25.0 & 0.9636 & 0.7969 & 61,617.53 \\
			\bottomrule
		\end{tabular}
	}%
\end{table}

\begin{table}[!ht]
	\centering
	\caption{
		Performance of k-means on Gaussian blob data with
		$n=100{,}000$, feature dimension $r=128$, and $k=100$ clusters.
		The table compares different distance kernels for $\rho\in\{5,10,50\}$. 
	}\label{tab:km-gaussblob3}
	\renewcommand{\arraystretch}{1.12}\resizebox{0.8\linewidth}{!}{%
		\begin{tabular}{l r r r r r r r}
			\toprule
			Method & $\rho$ & Time (s) & Speedup & Iter & AMI & ARI & SSE \\
			\midrule
			BF16 & - & 1.76 & 0.22 & 239.5 & 0.9721 & 0.8329 & 2,060,258.19 \\
			FP16 & - & 2.41 & 0.16 & 300.0 & 0.9584 & 0.7765 & 2,396,060.31 \\
			FP32 & - & 0.30 & 1.30 & 37.2 & 0.9567 & 0.7706 & 2,393,520.12 \\
			FP64 & - & 0.39 & 1.00 & 35.8 & 0.9567 & 0.7706 & 2,393,520.47 \\
			\midrule
			\multirow{3}{*}{BF16$\rightarrow$FP32}
			 & 5 & 0.94 & 0.41 & 35.8 & 0.9567 & 0.7706 & 2,393,520.19 \\
			 & 10 & 0.94 & 0.41 & 35.8 & 0.9567 & 0.7706 & 2,393,520.19 \\
			 & 50 & 0.94 & 0.41 & 35.8 & 0.9567 & 0.7706 & 2,393,520.19 \\
			\cdashline{1-8}
			\multirow{3}{*}{FP16$\rightarrow$FP32}
			 & 5 & 0.36 & 1.08 & 40.5 & 0.9567 & 0.7706 & 2,393,519.94 \\
			 & 10 & 0.95 & 0.41 & 35.8 & 0.9567 & 0.7706 & 2,393,520.19 \\
			 & 50 & 0.94 & 0.41 & 35.8 & 0.9567 & 0.7706 & 2,393,520.19 \\
			\midrule
			\multirow{3}{*}{BF16$\rightarrow$FP64}
			 & 5 & 1.09 & 0.36 & 43.5 & 0.9569 & 0.7711 & 2,393,569.38 \\
			 & 10 & 1.09 & 0.36 & 43.5 & 0.9569 & 0.7711 & 2,393,569.38 \\
			 & 50 & 1.09 & 0.36 & 43.5 & 0.9569 & 0.7711 & 2,393,569.38 \\
			\cdashline{1-8}
			\multirow{3}{*}{FP16$\rightarrow$FP64}
			 & 5 & 0.31 & 1.26 & 36.0 & 0.9567 & 0.7706 & 2,393,521.31 \\
			 & 10 & 0.80 & 0.49 & 32.0 & 0.9568 & 0.7707 & 2,393,521.81 \\
			 & 50 & 0.79 & 0.49 & 32.0 & 0.9568 & 0.7707 & 2,393,521.81 \\
			\cdashline{1-8}
			\multirow{3}{*}{FP32$\rightarrow$FP64}
			 & 5 & 0.29 & 1.34 & 35.0 & 0.9567 & 0.7706 & 2,393,520.75 \\
			 & 10 & 0.29 & 1.34 & 35.0 & 0.9567 & 0.7706 & 2,393,520.75 \\
			 & 50 & 0.29 & 1.34 & 35.0 & 0.9567 & 0.7706 & 2,393,520.75 \\
			\bottomrule
		\end{tabular}
	}%
\end{table}

\begin{table}[!ht]
	\centering
	\caption{
		Performance of k-means on Gaussian blob data with
		$n=100{,}000$, feature dimension $r=128$, and $k=500$ clusters.
		The table compares different distance kernels for $\rho\in\{5,10,50\}$. 
	}\label{tab:km-gaussblob4}
	\renewcommand{\arraystretch}{1.12}\resizebox{0.8\linewidth}{!}{%
		\begin{tabular}{l r r r r r r r}
			\toprule
			Method & $\rho$ & Time (s) & Speedup & Iter & AMI & ARI & SSE \\
			\midrule
			BF16 & - & 2.40 & 0.15 & 174.8 & 0.9689 & 0.7759 & 2,244,820.56 \\
			FP16 & - & 2.42 & 0.14 & 111.8 & 0.9581 & 0.7301 & 2,496,745.44 \\
			FP32 & - & 0.29 & 1.21 & 13.5 & 0.9572 & 0.7261 & 2,508,462.62 \\
			FP64 & - & 0.35 & 1.00 & 13.8 & 0.9572 & 0.7261 & 2,508,462.14 \\
			\midrule
			\multirow{3}{*}{BF16$\rightarrow$FP32}
			 & 5 & 1.47 & 0.24 & 13.8 & 0.9572 & 0.7261 & 2,508,462.19 \\
			 & 10 & 1.47 & 0.24 & 13.8 & 0.9572 & 0.7261 & 2,508,462.19 \\
			 & 50 & 1.47 & 0.24 & 13.8 & 0.9572 & 0.7261 & 2,508,462.19 \\
			\cdashline{1-8}
			\multirow{3}{*}{FP16$\rightarrow$FP32}
			 & 5 & 0.27 & 1.30 & 13.8 & 0.9576 & 0.7277 & 2,487,734.81 \\
			 & 10 & 1.66 & 0.21 & 13.8 & 0.9572 & 0.7261 & 2,508,462.19 \\
			 & 50 & 1.47 & 0.24 & 13.8 & 0.9572 & 0.7261 & 2,508,462.19 \\
			\midrule
			\multirow{3}{*}{BF16$\rightarrow$FP64}
			 & 5 & 1.34 & 0.26 & 13.0 & 0.9573 & 0.7265 & 2,499,024.38 \\
			 & 10 & 1.34 & 0.26 & 13.0 & 0.9573 & 0.7265 & 2,499,024.38 \\
			 & 50 & 1.34 & 0.26 & 13.0 & 0.9573 & 0.7265 & 2,499,024.38 \\
			\cdashline{1-8}
			\multirow{3}{*}{FP16$\rightarrow$FP64}
			 & 5 & 0.26 & 1.35 & 13.8 & 0.9576 & 0.7277 & 2,487,734.50 \\
			 & 10 & 1.59 & 0.22 & 13.8 & 0.9573 & 0.7264 & 2,504,334.19 \\
			 & 50 & 1.40 & 0.25 & 13.8 & 0.9573 & 0.7264 & 2,504,334.19 \\
			\cdashline{1-8}
			\multirow{3}{*}{FP32$\rightarrow$FP64}
			 & 5 & 0.24 & 1.46 & 12.2 & 0.9572 & 0.7261 & 2,508,463.19 \\
			 & 10 & 0.24 & 1.46 & 12.2 & 0.9572 & 0.7261 & 2,508,463.19 \\
			 & 50 & 0.24 & 1.46 & 12.2 & 0.9572 & 0.7261 & 2,508,463.19 \\
			\bottomrule
		\end{tabular}
	}%
\end{table}
\end{revision}